\documentclass[12pt]{amsart}
\usepackage{
  graphicx,amssymb,epic,eepic,picins,color,hyphenat,
  multicol, floatflt, slashbox, dbnsymb,
}
\usepackage[all]{xy}

\title{On the Kontsevich Integral for Knotted Trivalent Graphs}

\author{Zsuzsanna Dancso}
\address{
  Department of Mathematics\\
  University of Toronto\\
  Toronto Ontario M5S 2E4\\
  Canada
}
\email{zsuzsi@math.toronto.edu}
\urladdr{http://www.math.toronto.edu/\~{}zsuzsi}

\date{May 5, 2010}

\newtheorem{theorem}{Theorem}[section]
\newtheorem{lemma}[theorem]{Lemma}
\newtheorem{proposition}[theorem]{Proposition}
\newtheorem{corollary}[theorem]{Corollary}

\newtheorem{remark}{Remark}

\begin{document}
\newdimen\captionwidth\captionwidth=\hsize

\begin{abstract}
In this paper we construct an extension of the Kontsevich integral
of knots to knotted trivalent graphs, which commutes with orientation switches, 
edge deletions, edge unzips, and connected sums.
In 1997 Murakami and Ohtsuki \cite{murakami} first constructed 
such an extension, building on Drinfel'd's theory of associators. 
We construct a step by step definition, using elementary 
Kontsevich integral methods, to get
a one-parameter family of corrections that all yield invariants well
behaved under the graph operations above.
\end{abstract}

\maketitle

\section{Introduction}

The goal of this paper is to construct an extension of the Kontsevich
integral $Z$ of knots to knotted trivalent graphs. The extension is 
a universal finite type invariant of knotted trivalent graphs, 
which commutes with natural operations that are defined on the
space of graphs, as well as on the target space of $Z$. These operations are
changing the orientation of an edge; deleting an edge; unzipping an edge 
(an analogue of cabling of knots); and connected sum.

One reason this is interesting is that several knot
properties (such as genus, unknotting number and ribbon property, for example) are
definable with short formulas involving knotted trivalent graphs
and the above operations. Therefore, such an operation-respecting
invariant yields algebraic necessary conditions for these properties,
i.e. equations in the target space of the invariant. This idea is
due to Dror Bar-Natan, and is described in more detail in \cite{dbn2}.
The extension of $Z$ is the first example for such an invariant. Unfortunately,
the target space of Z is too complicated for it to be useful in a computational
sense. However, we hope that by finding sufficient quotients of the
target space more computable invariants could be born.

The construction also provides an algebraic description of the
Kontsevich integral (of knots and graphs), 
due to the fact that knotted trivalent
graphs are finitely generated, i.e. there's a finite (small) set of 
graphs such that any knotted trivalent graph can be obtained from
these using the above operations. This is described in more detail in
\cite{thurst}. Since the extension commutes with the operations, it is 
enough to compute it for the graphs in the generating set. As knots are
special cases of knotted trivalent graphs, this also yields an algebraic
description of the Kontsevich integral of knots.

The Kontsevich invariant $Z$ was first extended to knotted trivalent graphs by Murakami and
Ohtsuki in \cite{murakami}, and by Cheptea and Le in \cite{chele}.
Both papers extend the combinatorial definition of $Z$, using q-tangles 
(a.k.a. parenthesized tangles) and building on a significant body of 
knowledge about Drinfel'd's associators to prove that the extension 
is a well-defined invariant. 
When one tries to extend $Z$ naively,
replacing knots by knotted trivalent graphs, the result is
neither convergent nor an isotopy invariant. Thus, one needs to apply
re-normalizations to make it converge, and corrections to make it
invariant. Using q-tangles, \cite{murakami} and \cite{chele} do
not have to deal with the convergence issue, while similar invariance 
issues arise in both approaches.

The extension constructed in \cite{chele} uses even associators. Cheptea and 
Le proved that their construction is an isotopy 
invariant (in part building on \cite{murakami}), and that it commutes 
with orientation switches and edge deletions and is unique if certain local properties
are required -in this sense it is the strongest. They also conjecture that it coincides 
with Murakami and Ohtsuki's construction, which was later confirmed. 

The main purpose of this
paper is to eliminate the black box quality of the extended invariants,
which results partly from the depth of the ingredients that go into
them, and partly from the fact that the proofs needed for the
constructions are spread over several papers 
(\cite{murakami}, \cite{lmmo}, \cite{lemu} or \cite{chele}, \cite{murakami}).

Our construction differs from previous ones in that we build
the corrected extension step by step on the naive one. After re-normalizations 
to make the extension convergent, non-invariance errors arise. 
We fix some of these by introducing counter terms (corrections)
that are precisely the inverses of the errors, and we show that thanks
to some ``syzygies'', i.e. dependencies between the errors, all the
other errors get corrected automatically. 
The proofs involve mainly elementary Kontsevich integral methods and 
combinatorial considerations.

As a result, we get a series of different corrections that all
yield knotted graph invariants. One of these is the Murakami-Ohtsuki
invariant.

We note that as a special case, our construction also produces 
an associator.

\smallskip

Since the construction has many details and there's a risk of the main ideas 
getting lost among them, Section 2 is an ``Executive summary'' of the key 
points and important steps. All the details are omitted here, and follow
later in the paper. The purpose of this section is to emphasize what is
important and to provide an express lane to those familiar with the topic.

In an effort to make the construction and the inner workings of $Z$
as transparent as possible, Section 3 is dedicated to reviewing the relevant results
for the Kontsevich integral of knots. We state all the results that
we need, mostly with proofs. The reason for reproducing the proofs is that
we need to modify some of them for the results to carry over to graphs,
therefore understanding the knot case is crucial for the readability of the paper.
Our main reference is a 
nice exposition by Chmutov and Duzhin, \cite{chmutov}. This is the source 
of any theorems and proofs in Section 2, unless otherwise stated. We also
use results from Bar-Natan's paper \cite{dbn},
which is another good reference for the Kontsevich integral of knots in general. Other
references include Kontsevich's original paper \cite{konts}.

Section 4 is dedicated to defining the necessary framework for the extension
and trying (and failing) to naively extend $Z$ to graphs.

In Section 5 we use a baby version of a renormalization technique from quantum 
field theory
to eliminate the divergence that occurs in the case of the naive extension,
and prove that the resulting not-quite-invariant has some promising properties.

The bulk of the difficulty lies in Section 6, where we need to find the
appropriate correction factors to make the extension an isotopy invariant.
This involves some computations and a series of combinatorial considerations, 
but is done in a fairly elementary way overall.
It will turn out that the result will almost, but not quite commute
with the unzip operation (and this will happen for any invariant
that commutes with edge deletion), so we need to re-normalize the unzip
operation to get a fully well behaved invariant.

\subsection{Acknowledgments} I am greatly indebted to my advisor, Dror Bar-Natan, 
for suggesting this project to me, weekly helpful discussions, and proofreading
this paper several times. I also wish to express my gratitude to the referee for the
thorough reading of the paper and many detailed suggestions.

\section{Executive summary}
This section contains the main points of the construction, but no
details. All the details will follow later in the paper.

Trivalent graphs are graphs with three vertices joining at each vertex.
We allow multiple edges, loops, and circles with no vertices.
Knotted trivalent graphs are embeddings of trivalent graphs into 
${\mathbb R}^3$, modulo isotopy. We require all edges to be oriented 
and framed. In all figures we will use the blackboard framing convention
(normal vectors pointing at the reader).

There are four operations of knotted trivalent graphs: orientation switch,
edge deletion, edge unzip, and connected sum. When deleting an edge, two
vertices cease to exist. Unzip is an analogue of cabling, as shown:
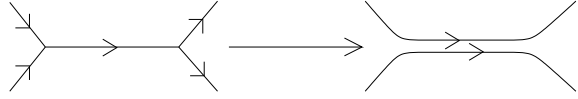
\begin{center}
\begin{picture}(0,0)%
\includegraphics{unzip1.pstex}%
\end{picture}%
\setlength{\unitlength}{2329sp}%
\begingroup\makeatletter\ifx\SetFigFont\undefined%
\gdef\SetFigFont#1#2#3#4#5{%
  \reset@font\fontsize{#1}{#2pt}%
  \fontfamily{#3}\fontseries{#4}\fontshape{#5}%
  \selectfont}%
\fi\endgroup%
\begin{picture}(6051,984)(1789,-913)
\end{picture}%

\end{center}
Connected sum depends on choices of edges to connect, and produces two 
new vertices:
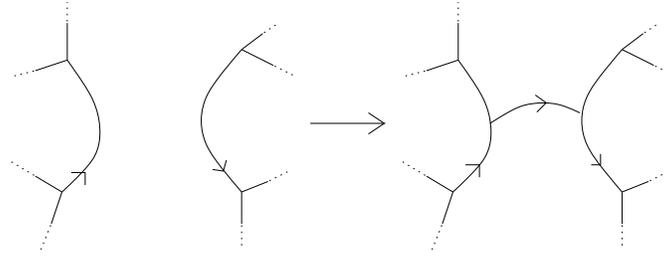
\begin{center}
\begin{picture}(0,0)%
\includegraphics{connsum1.pstex}%
\end{picture}%
\setlength{\unitlength}{2565sp}%
\begingroup\makeatletter\ifx\SetFigFont\undefined%
\gdef\SetFigFont#1#2#3#4#5{%
  \reset@font\fontsize{#1}{#2pt}%
  \fontfamily{#3}\fontseries{#4}\fontshape{#5}%
  \selectfont}%
\fi\endgroup%
\begin{picture}(6458,2424)(964,-2098)
\end{picture}%

\end{center}

The Kontsevich integral of knots is defined by the following
integral formula:
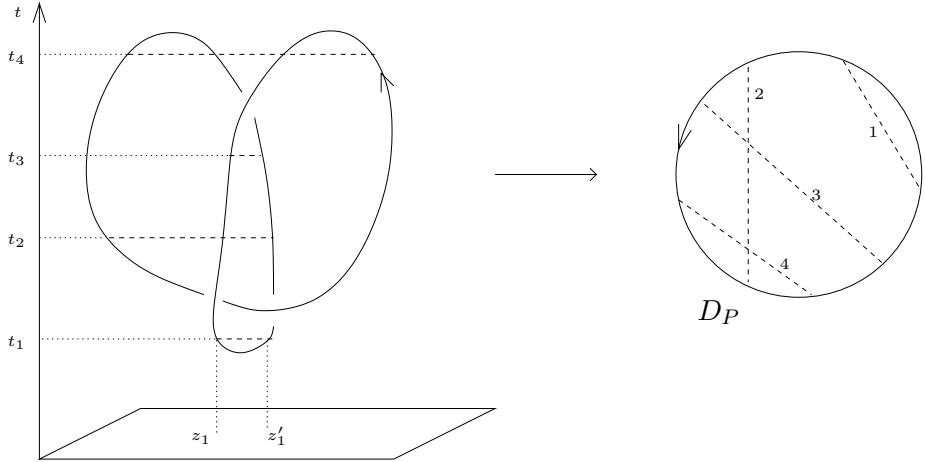
\begin{center}
\input kintegral.pstex_t
\end{center}

$$
Z(K)=\sum_{m=0}^{\infty} \thinspace
\mathop{\int\limits_{t_{min}<t_1<...<t_m<t_{max}}}_{t_i \thinspace non-critical} 
\sum_{P=(z_i,z_{i}')}
\frac{(-1)^{P_{\downarrow}}}{(2\pi i)^m}
D_P \bigwedge_{i=1}^{m} \frac{dz_i-dz_i'}{z_i-z_i'}
$$

We naively generalize this to knotted trivalent graphs (or trivalent tangles,
defined precisely in Section 4.1), 
by simply putting graphs in the picture and applying the same formula.

The integral now takes its values in chord diagrams on trivalent graph ``skeletons''.
These are, just like ordinary chord diagrams, factored out by the $4T$ relation,
and one additional relation, called vertex invariance, or VI:
\begin{center}
\input vi.pstex_t
,
\end{center}
where the sign $(-1)^{\rightarrow}$ is $-1$ if the edge the chord is ending
on is oriented to be outgoing from the vertex, and $+1$ if it's incoming.

We do not mod out by the one term relation, since we don't want the
invariant to be framing-independent.

The naively defined extension has promising properties: it preserves the
factorization property (multiplicativity) of the ordinary Kontsevich integral, 
and it is invariant under horizontal
deformations that leave the critical points and vertices fixed. It holds promise
to be a universal finite type invariant of knotted trivalent graphs. But,
unfortunately, it is divergent.

\smallskip

The divergence is caused by ``short chords'' near the critical points, 
i.e. chords that are not separated from the critical point by another
chord ending -- 
since we didn't factor out
by the one term relation; and short chords near the vertices
(which cause a divergence the same way, but we don't have any reason to factor
them out).

The Kontsevich Integral has been previously extended to framed links (and
framed tangles) by Le and Murakami in \cite{lemu2}, \cite{lemu3}, and by Goryunov
in \cite{gor}. We use essentially the same method as Le and Murakami, which extends 
to trivalent vertices easily. Goryunov's approach is different, using an 
$\varepsilon$-shift of the knot along a general framing (not necessarily the
blackboard framing).

The technique we use (and which \cite{lemu2}, \cite{lemu3} used) is borrowed 
from quantum field theory: we know exactly what the divergence is, so we 
``multiply by its inverse''.

We choose a fixed scale $\mu$ and open up the strands at each vertex and critical 
point to 
width $\mu$, at a small distance $\varepsilon$ from the vertex or critical point, 
as shown for a vertex of a $\lambda$ shape:
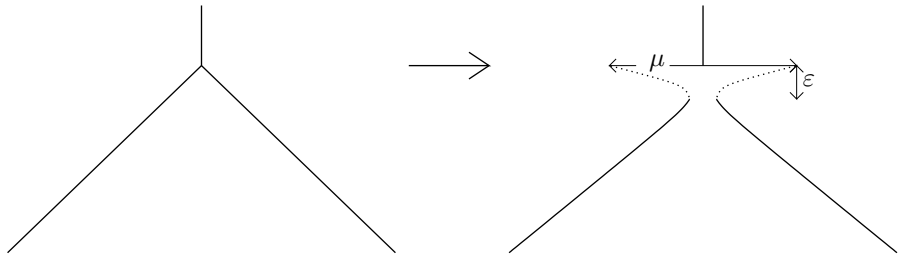
\begin{center}
\input renormalize.pstex_t
\end{center}
We compute the integral using this picture instead of the original one.
We only allow chords between the dotted parts, but not chords coming
from far away ending on them. The renormalized integral is the limit of this
quantity as $\varepsilon$ tends to zero. We do the same for vertices of a $Y$
shape, and critical points.

We will have to insist that we embed the graph in a way that only
contains vertices of either $\lambda$ or $Y$ shapes
(I.e, either one edge is locally above the vertex and two 
are below, creating a $\lambda$-shape, or two are above and one below, 
creating a $Y$-shape. We do not allow a vertex that is a local
minimum or maximum at the same time.) There is of course
such an embedding in any isotopy class, however, this will cause some 
invariance issues later.

The re-normalized version is convergent (this is easy to prove), it retains
the good properties of the naive extension, and it is invariant under rigid motions
of critical points and vertices (i.e. deformations that do not change the 
number of critical points or the $\lambda$/$Y$ type of the vertices).

There are natural orientation switch, edge delete, edge unzip, and connected
sum operations defined on the chord diagrams on graph skeletons, and the
re-normalized integral commutes with all these operations. 

Furthermore, it has sensible behavior under changing the scale $\mu$: an
easy computation shows that the value of $Z$ will get multiplied by a simple
exponential factor (a chord diagram on two strands) at all vertices 
and critical points.

\smallskip

The problem with the re-normalized integral is that it is far from being
an isotopy invariant: it does not tolerate deformations changing the number of
critical points (``straightening humps''- the same problem arises in
case of the ordinary Kontsevich integral), or changing the shape of the vertices.

Similar to the knot case, but more complicated, we will have to introduce
corrections to make $Z$ an invariant. There are four ``correctors'' available:
we can put prescribed little chord diagrams on each minimum, maximum, 
$\lambda$-vertex, and $Y$-vertex. We will call these $u$, $n$, 
$\lambda$ and $Y$.

There are eight moves needed for the isotopy invariance of $Z$,
but it can be shown, using several relations (syzygies), that the following
three suffice (these will be called moves $1$, $3$ and $4$ later):
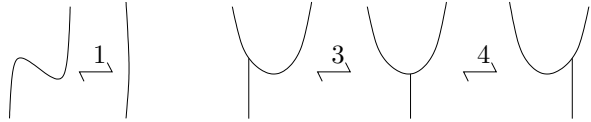
\begin{center}
\input moves1.pstex_t
\end{center}

The moves translate to equations between the correctors $u$, $n$, $\lambda$
and $Y$.

We know how to solve the equation corresponding to move $1$, as this was done 
even in the knot case.

\smallskip

The most difficult step is solving equations $3$ and $4$. These are
obviously not independent, since the leftmost and rightmost sides are
vertical mirror images, and the Kontsevich integral of a vertical mirror image is the 
vertical mirror image. However, it is not true that any set of corrections that
fix move $3$ will fix move $4$ automatically (i.e. there's no missing 
syzygy). What we need to do is to solve the two equations simultaneously.

We achieve this by showing that moves $3$ and $4$ are equivalent to
the following equations of chord diagrams:
$$\raisebox{-0.2 in}{\input pic20.pstex_t} \thickspace = \thickspace
\raisebox{-0.2 in}
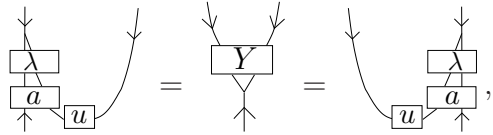
{\begin{picture}(0,0)%
\includegraphics{yvertex.pstex}%
\end{picture}%
\setlength{\unitlength}{1421sp}%
\begingroup\makeatletter\ifx\SetFigFont\undefined%
\gdef\SetFigFont#1#2#3#4#5{%
  \reset@font\fontsize{#1}{#2pt}%
  \fontfamily{#3}\fontseries{#4}\fontshape{#5}%
  \selectfont}%
\fi\endgroup%
\begin{picture}(1464,2289)(1159,-2473)
\put(1711,-1381){\makebox(0,0)[lb]{\smash{{\SetFigFont{12}{14.4}{\rmdefault}{\mddefault}{\updefault}{\color[rgb]{0,0,0}$Y$}%
}}}}
\end{picture}%
} \thickspace = \thickspace
\raisebox{-0.2 in}{\input pic21.pstex_t} \medspace ,$$
where we ``compute'' $a$ explicitly, i.e. we express it as the 
re-normalized Kontsevich integral of the simple tangle 
\raisebox{-2mm}{\begin{picture}(0,0)%
\includegraphics{adef.pstex}%
\end{picture}%
\setlength{\unitlength}{1421sp}%
\begingroup\makeatletter\ifx\SetFigFont\undefined%
\gdef\SetFigFont#1#2#3#4#5{%
  \reset@font\fontsize{#1}{#2pt}%
  \fontfamily{#3}\fontseries{#4}\fontshape{#5}%
  \selectfont}%
\fi\endgroup%
\begin{picture}(1224,699)(1189,-1348)
\end{picture}%
}.

This is a fairly elementary, but tricky computation.

All we use about $a$ for the invariance though is that it is 
mirror symmetric, which is obvious from its definition. (We use 
more of its properties later.)

Looking at the above equations there's an obvious set of corrections
that will make $Z$ an isotopy invariant:
$\lambda=a^{-1}$, $Y=1$, $u=1$, $n=\nu^{-1}$, where $n$ is determined by 
$u$ the same way it is in the knot case. 
The element $\nu$ is defined by $Z(\raisebox{-1.5 mm}{\begin{picture}(0,0)%
\includegraphics{hump.pstex}%
\end{picture}%
\setlength{\unitlength}{3947sp}%
\begingroup\makeatletter\ifx\SetFigFont\undefined%
\gdef\SetFigFont#1#2#3#4#5{%
  \reset@font\fontsize{#1}{#2pt}%
  \fontfamily{#3}\fontseries{#4}\fontshape{#5}%
  \selectfont}%
\fi\endgroup%
\begin{picture}(154,249)(964,-373)
\end{picture}%
})$.
Note that what we denote by $\nu$ is denoted by $\nu^{-1}$ in much of 
the literature (\cite{lmmo}, \cite{murakami}, \cite{ohts}, \cite{lmo}),
and that $\nu^{-1}$ is also known as the invariant of the unknot. 
Some attention needs to be paid to edge
orientations, to make sure $Z$ commutes with orientation switches.

To show that the resulting invariant commutes with edge deletion and connected 
sum, we use another property of $a$ that is almost obvious from its
expression as a value of $Z$.

By rearranging the equations, we produce a one-parameter family of corrections
all yielding isotopy invariants, one of these is the set of corrections used by 
Murakami and Ohtsuki.

Finally, we show that unfortunately, no invariant will commute at the same
time with the edge delete and edge unzip operations, as they are defined, so
we re-normalize the unzip operation to make $Z$ fully well-behaved. This amounts to
observing what the error is and multiplying by its inverse.

\section{A quick overview of the Kontsevich integral for knots}

The reference for everything in this section, unless otherwise stated,
is \cite{chmutov}.

\subsection{Finite type invariants and the algebra ${\mathcal A}$}

The theory of finite type (or Vassiliev) invariants grew out of the idea
of V. Vassiliev to extend knot invariants to the class of singular knots.
By a singular knot we mean a knot with a finite number of simple (transverse)
double points. The extension of a
${\mathbb C}$-valued knot invariant $f$ follows the rule

$$ f(\doublepoint)=f(\overcrossing)-f(\undercrossing).$$

A {\it finite type (Vassiliev) invariant} is a knot invariant
whose extension vanishes on all knots with more than $n$ double points,
for some $n \in {\mathbb N}$. The smallest such $n$ is called the {\it order},
or {\it type} of the invariant.

\bigskip

The set of all Vassiliev invariants forms a vector space ${\mathcal V}$,
which is filtered by the subspaces ${\mathcal V}_n$, the
Vassiliev invariants of order at most $n$:

$${\mathcal V}_0 \subseteq {\mathcal V}_1 \subseteq {\mathcal V}_2
\subseteq ... \subseteq {\mathcal V}_n \subseteq ... \subset {\mathcal V}$$

This filtration allows us to study the simpler associated graded space:

$$ gr{\mathcal V}={\mathcal V}_0 \oplus {\mathcal V}_1/{\mathcal V}_0
\oplus {\mathcal V}_2/{\mathcal V}_1 \oplus ... \oplus
{\mathcal V}_n/{\mathcal V}_{n-1} \oplus ...$$

The components ${\mathcal V}_n/{\mathcal V}_{n-1}$ are best understood in
terms of {\it chord diagrams}.

\bigskip

A {\it chord diagram} of order $n$ is an oriented circle with a set of n {\it chords}
all of whose endpoints are distinct points of the circle (see the figure below). The actual 
shape of the chords and the exact position of endpoints are
irrelevant, we are only interested in the pairing they define on the $2n$
cyclically ordered points:

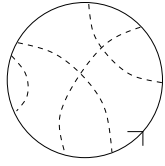
\begin{center}
\begin{picture}(0,0)%
\includegraphics{chorddiag.pstex}%
\end{picture}%
\setlength{\unitlength}{2131sp}%
\begingroup\makeatletter\ifx\SetFigFont\undefined%
\gdef\SetFigFont#1#2#3#4#5{%
  \reset@font\fontsize{#1}{#2pt}%
  \fontfamily{#3}\fontseries{#4}\fontshape{#5}%
  \selectfont}%
\fi\endgroup%
\begin{picture}(1823,1822)(515,-1047)
\end{picture}%

\end{center}

The {\it chord diagram of a singular knot} $S^1 \to S^3$ is the oriented
circle $S^1$ with the pre-images of each double point connected by a chord.

\bigskip

\smallskip
Let ${\mathcal C}_n$ be the vector space spanned by all chord diagrams of
order $n$.

Let $F_n$ be the vector space of all $\mathbb C$-valued linear functions on
${\mathcal C}_n$.

A Vassiliev invariant $f \in {\mathcal V}_n$ determines a function
$[f] \in F_n$ defined by $[f](C)=f(K)$, where K is any singular  knot whose
chord diagram is C.

The fact that $[f]$ is well defined (does not depend on the choice of $K$)
can be seen as follows: If $K_1$ and $K_2$ are two singular knots with the same
chord diagram, then they can be projected on the plane in such a way that
their knot diagrams coincide except possibly at a finite number of crossings,
where $K_1$ may have an over-crossing and $K_2$ an under-crossing or vice versa.
But since $f$ is an invariant of order $n$ and $K$ has $n$ double points, a
crossing flip does not change the value of $f$ (since the difference would
equal to the value of $f$ on an $(n+1)$-singular knot, i.e. 0).

\bigskip

The kernel of the map ${\mathcal V}_n \to F_n$ is, by definition,
${\mathcal V}_{n-1}$. Thus, what we have defined is an inclusion
$i_n: {\mathcal V}_n/{\mathcal V}_{n-1} \to F_n$. The image of this inclusion
(i.e. the set of linear maps on chord diagrams that come from Vassiliev
invariants) is described by two relations:

\bigskip

{\bf 4T, the four-term relation}:

\begin{center}
\input 4t.pstex_t
\end{center}
\noindent
for an arbitrary fixed position of $(n-2)$ chords (not drawn here) and
the two additional chords as shown.

This follows from the following fact of singular knots:

\begin{center}
\input knot4t.pstex_t
\end{center}
\noindent
which is easy to show using the following isotopy:

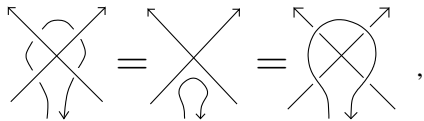
\begin{center}
\input trick.pstex_t
\end{center}
\noindent
(or by resolving both double points).

\smallskip

{\bf FI, the framing independence relation} (a.k.a. {\bf One-term} relation):

\begin{center}
\input fi.pstex_t
\end{center}
\noindent
which follows from the Reidemeister one move of knot diagrams.

(The dotted arcs here and on the following pictures mean that there might
be further chords attached to the circle, the positions of which are fixed
throughout the relation.)

\bigskip

To explain the name of this relation, let us say a word about framed knots.
A {\it framing} on a curve is a smooth choice of a normal vector at each point
of the curve, up to isotopy. This is equivalent to ``thickening'' the curve into a band, where
the band would always be orthogonal to the chosen normal vector. A knot projection
(knot diagram) defines a framing (``blackboard framing''), where we
always choose the normal vector that is normal to the plane that we project to.
Every framing (up to isotopy) can be represented as a blackboard framing, for some projection.

If we were to study framed knots, these are in one-to-one
correspondence to knot diagrams modulo Reidemeister moves $R2$ and $R3$, as $R1$
adds or eliminates a twist, i.e. it changes the framing.

The framing independence relation arises from $R1$, meaning that the knot invariants
we work with are independent of any framing on the knot. This will change later
in the paper as we turn to graphs.

\smallskip

We define the algebra ${\mathcal A}$, as a direct sum of the vector spaces
${\mathcal A}_n$ generated by chord diagrams of order $n$ considered
modulo the FI and 4T relations. 
By an abuse of notation, from now on we consider $4T$ and $FI$ to be 
relations in ${\mathcal C}_n$, i.e. define the appropriate (sums of)
chord diagrams to be zero.
The multiplication on ${\mathcal A}$ is
defined by the connected sum of chord diagrams, which is well defined
thanks to the 4T relations (for details, see for example \cite{dbn}).

\smallskip

The ${\mathbb C}$-valued linear functions on $\mathcal A_n$ are called
{\it weight systems of order n}. The above construction shows that every
Vassiliev invariant defines a weight system of the same order.

\smallskip

The famous theorem of Kontsevich, also known as the Fundamental Theorem
of Finite Type Invariants, asserts that every weight system arises as
the weight system of a finite type invariant. The proof relies on the
construction of a {\it universal finite type invariant}, called the
Kontsevich Integral $Z$, which takes its values in
the graded completion of ${\mathcal A}$. Given a weight system,
one gets the appropriate finite type invariant by pre-composing with $Z$.

\subsection{The definition of the Kontsevich Integral}

Let us represent ${\mathbb R}^3$ as a direct product of
a complex plane $\mathbb C$
with coordinate $z$ and a real line with coordinate $t$. Let the knot $K$
be embedded in ${\mathbb C} \times {\mathbb R}$ in such a way that the
coordinate $t$ is a Morse function on $K$.

The {\it Kontsevich integral} of $K$ is an element in the graded completion
of ${\mathcal A}$, defined by the following formula:

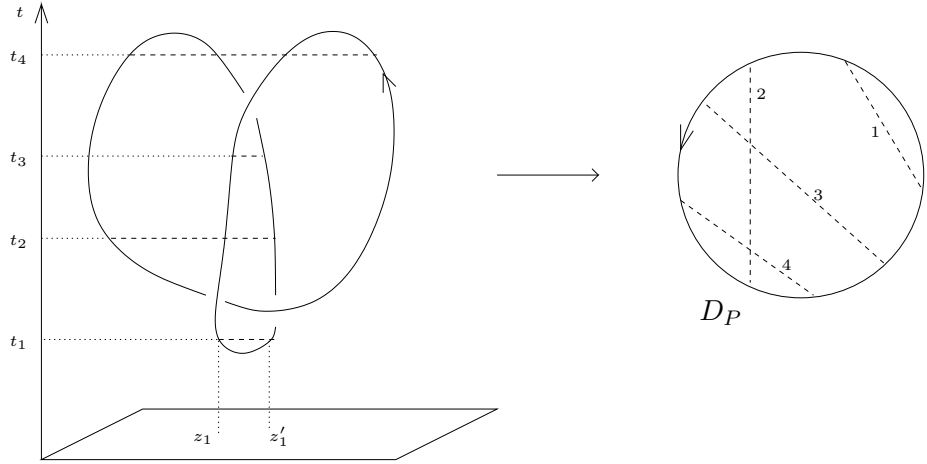
\begin{center}
\input kintegral.pstex_t
\end{center}

$$
Z(K)=\sum_{m=0}^{\infty} \thinspace
\mathop{\int\limits_{t_{min}<t_1<...<t_m<t_{max}}}_{t_i \thinspace non-critical} 
\sum_{P=(z_i,z_{i}')}
\frac{(-1)^{P_{\downarrow}}}{(2\pi i)^m}
D_P \bigwedge_{i=1}^{m} \frac{dz_i-dz_i'}{z_i-z_i'}
$$

In the formula, $t_{min}$ and $t_{max}$ are the minimum and maximum of
the function $t$ on $K$.

The integration domain is the $m$-dimensional simplex
$t_{min}<t_1<...<t_m<t_{max}$ divided by the critical values into a
number of connected components. The number of summands in the integral
is constant in each of the connected components.

In each plane $\{t=t_j\}$
choose an unordered pair of distinct points $(z_j,t_j)$ and $(z_j',t_j)$
on K, so that $z_j(t_j)$ and $z_j'(t_j)$ are continuous functions. $P$
denotes the set of such pairs for each $j$. The integrand is the sum
over all choices of $P$.

For a pairing $P$, $P_{\downarrow}$ denotes the
number of points $(z_j,t_j)$ or $(z_j',t_j)$ in P where $t$ decreases along
the orientation of $K$.

$D_P$ denotes the chord diagram obtained by joining
each pair of points in $P$ by a chord, as the figure shows.

Over each connected component, $z_j$ and $z_j'$ are smooth functions. By
$\bigwedge_{i=1}^{m} \frac{dz_i-dz_i'}{z_i-z_i'}$ we mean the pullback of
this form to the simplex.

The term of the Kontsevich integral corresponding to $m=0$ is, by
convention, the only chord diagram with no chords, with coefficient one,
i.e, the unit of the algebra ${\mathcal A}$. From now on, we will refer
to this element as $1 \in {\mathcal A}$.

\subsection{Convergence}

Let us review the proof of the fact that each integral in the above
formula is convergent.
By looking at the definition, one observes that the only way the integral
may not be finite is the $(z_i-z_i')$ in the denominator getting
arbitrarily small near the critical points (the boundaries of the
connected components of the integration domain).
This only happens near a local minimum or maximum in the knot - otherwise the
minimum distance between strands is a lower bound for the denominator.

If a chord $c_k$ is separated from the critical value by another
``long'' chord $c_{k+1}$ ending
closer to the critical value, as shown below, then the
smallness in the denominator corresponding to chord $c_k$ will be canceled
by the smallness of the integration domain for $c_{k+1}$, hence the integral
converges:

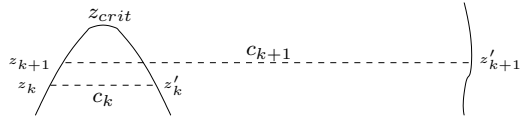
\begin{center}
\input nonisolated.pstex_t
\end{center}

The integral for the long chord can be estimated as follows (using the
picture's notation):

$$ \Big| \int_{t_k}^{t_{crit}} \frac{dz_{k+1}-dz_{k+1}'}{z_{k+1}-z_{k+1}'} \Big| \leq
C \Big| \int_{t_k}^{t_{crit}} d(z_{k+1}-z_{k+1}') \Big| =$$

$$C| (z_{crit}-z_{k})-(z_{k+1}(t_{crit})-z_{k+1}'(t_k))| \leq C'|z_k-z_k'| $$

For some constants C and C'. So the integral for the long chord is as small
as the denominator for the short chord, therefore the integral converges.

\bigskip

Thus, the only way a divergence can occur is the case of an isolated chord,
i.e. a chord near a critical point that is not separated from it by any
other chord ending. But, by the one term relation, chord diagrams
containing an isolated chord are declared to be zero, which makes the
divergence of the corresponding integral a non-issue.

\subsection{Invariance}

Since horizontal planes cut the knot into tangles,
we will use tangles and their properties
to prove the invariance of the Kontsevich integral in the class of Morse
knots.

By a {\it tangle} we mean a 1-manifold embedded in $[0,1]^3$, whose boundary
is the union of $k$ points on the bottom face of the cube, positioned at 
$$(\frac{1}{k+1},\frac{1}{2},0),...,(\frac{k}{k+1},\frac{1}{2},0);$$ 
and $l$ points on the top face, positioned at
$$(\frac{1}{l+1},\frac{1}{2},1),...,(\frac{l}{l+1},\frac{1}{2},1).$$
Two tangles are considered equal if there is an isotopy of the cubes
that fixes their boundary and takes one tangle to the other.

Tangles can be multiplied by stacking one cube on top of another and rescaling,
if the number of endpoints match.

\bigskip

A {\it tangle chord diagram} is a tangle supplied with a set of {\it horizontal}
chords considered up to a diffeomorphism of the tangle that preserves
the horizontal fibration.

Multiplication of tangles induces a multiplication of tangle chord diagrams in 
the obvious way.

If $T$ is a tangle, the space ${\mathcal A}_T$ is a vector space generated
by all chord diagrams on $T$, modulo the set of {\it tangle one-} and
{\it four-term relations}:

\bigskip

{\bf The tangle one-term relation} (or framing independence): A tangle
chord diagram with an isolated chord is equal to zero in ${\mathcal A}_T$.

\bigskip

{\bf The tangle $4T$ relation}: Consider a tangle consisting of $n$
parallel vertical strands. Denote by $t_{ij}$ the chord diagram with
a single horizontal chord connecting the $i$-th and $j$-th strands,
multiplied by $(-1)^\downarrow$, where $\downarrow$ stands for the
number of endpoints of the chord lying on downward-oriented strands.

\begin{center}
\input tij.pstex_t
\end{center}

The tangle $4T$ relation can be expressed as a commutator in terms of
the $t_{ij}$'s:

$$[t_{ij}+t_{ik},t_{jk}]=0.$$

One can check that by closing the three vertical strands into a circle
respecting their orientations, the tangle $4T$ relation carries over into
the ordinary $4T$ relation.

\smallskip

We take the opportunity here to mention a useful lemma, a slightly
different version of which appears in D. Bar-Natan's paper \cite{dbn},
and a special case is stated in Murakami and Ohtsuki, \cite{murakami}.
This is a direct consequence of the tangle $4T$ relations:

\begin{lemma} \label{loc}
{\bf Locality.}
Let $T$ be the tangle consisting of $n$ parallel vertical strands, and $D$ be any
chord diagram such that no chords end on the $n$-th string. Let $S$ be
the sum $\sum\limits_{i=1}^{n-1}t_{in}$ in ${\mathcal A}_T$. Then $S$ commutes
with D in ${\mathcal A}_T$.
\end{lemma}

\smallskip

The Kontsevich integral is defined for tangles the same way it is defined
for knots.

By Fubini's theorem, it is multiplicative:
$$Z(T_1)Z(T_2)=Z(T_1 T_2),$$
whenever the product $T_1 T_2$ is defined.

This implies the important
fact that the Kontsevich integral of the (vertical, and by the
invariance results, any) connected sum of knots
is the product (in the algebra ${\mathcal A}$ of chord diagrams) of the
Kontsevich integrals of the summands. By the invariance results this
will generalize to any connected sum. We will sometimes refer to
this as the factorization property, or multiplicativity.

\begin{proposition} \label{hor}
The Kontsevich integral is invariant under horizontal deformations
(deformations preserving the $t$ coordinate) of
the knot that leave the critical points fixed.
\end{proposition}

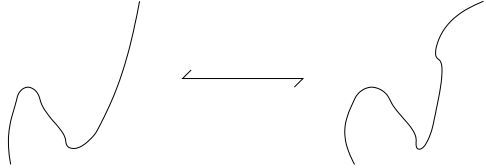
\begin{center}
\begin{picture}(0,0)%
\includegraphics{horizdef.pstex}%
\end{picture}%
\setlength{\unitlength}{2842sp}%
\begingroup\makeatletter\ifx\SetFigFont\undefined%
\gdef\SetFigFont#1#2#3#4#5{%
  \reset@font\fontsize{#1}{#2pt}%
  \fontfamily{#3}\fontseries{#4}\fontshape{#5}%
  \selectfont}%
\fi\endgroup%
\begin{picture}(4168,1449)(720,-1048)
\end{picture}%

\end{center}

\begin{proof}
Let us decompose the knot into a product of tangles without critical
points, and other (``thin'') tangles containing one unique critical point.

The following lemma addresses the case of tangles without critical
points.
The proposition then follows from the lemma by taking a limit.
(See \cite{chmutov} for more details.)
\end{proof}

\begin{lemma} \label{braidhor}
Let $T_0$ be a tangle without critical points and $T_{\lambda}$ a horizontal
deformation of $T_0$ into $T_1$, such that $T_{\lambda}$ fixes the top and
the bottom of the tangle. Then $Z(T_0)=Z(T_1)$.
\end{lemma}

We will use this lemma as an ingredient without any modification, so
we omit the proof here, which uses Stokes Theorem and the fact
that the differential form inside the integral is exact.  Details 
can be found in \cite{chmutov}, for example.

The next lemma is the only one where we go into more detail than
Chmutov and Duzhin in \cite{chmutov}. We will need a
modification of this proof in the graph case, so we felt it was important
to use rigorous notation and touch on the fine points.

\begin{proposition} \label{mvcrits}
{\bf Moving critical points.} Let $T_0$ and $T_1$ be two tangles that
differ only in a
thin needle (possibly twisted), as in the figure,
such that each level $\{t=c\}$ intersects the
needle in at most two points and the distance between these is at most
$\varepsilon$. Then $Z(T_0)=Z(T_1)$.
\end{proposition}

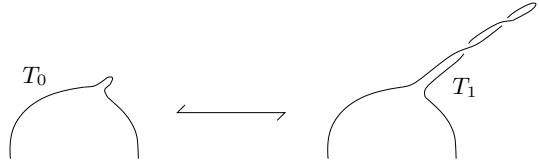
\begin{center}
\input needle.pstex_t
\end{center}

\begin{proof}
$Z(T_0)$ and $Z(T_1)$ can only differ in terms with a chord ending on the
needle. If the chord closest to the end of the needle connects
the two sides of the needle (isolated chord), then the corresponding
diagram is zero by the FI (1T) relation.

So, we can assume that the one closest to the needle's end is a ``long
chord'', suppose the endpoint belonging to the needle is $(z_k,t_k)$.
Then, there's another choice for the $k$-th chord which touches the needle
at the opposite point $(z_k'',t_k)$, as the figure shows, and $D_P$ will
be the same for these two choices.

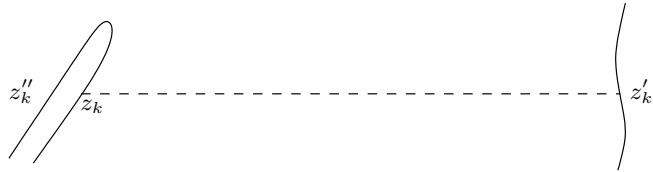
\begin{center}
\input pairedchord.pstex_t
\end{center}

The corresponding two terms appear in $Z(T_1)$ with opposite signs due
to $(-1)^{P_{\downarrow}}$, and the difference of the integrals
can be estimated as follows:

$$\Big| \int_{t_{k-1}}^{t_c} d(ln(z_k-z_k')) - \int_{t_{k-1}}^{t_c}
d(ln(z_k''-z_k')) \Big| = \Big| ln \Big( \frac{z_{k}''(t_{k-1})-z_{k}'(t_{k-1})}
{z_{k}(t_{k-1})-z_{k}'(t_{k-1})} \Big) \Big| = $$

$$= \Big| ln\Big(1+ \frac{z_{k}''(t_{k-1})-z_{k}(t_{k-1})}
{z_{k}(t_{k-1})-z_{k}'(t_{k-1})}\Big) \Big| \leq
C|z_{k}''(t_{k-1})-z_{k}(t_{k-1})| \leq C\varepsilon, $$

where $t_c$ is the value of $t$ at the tip of the needle,
and C is a constant depending on the minimal distance of
the needle to the rest of the knot.

\smallskip

If the next, $(k-1)$-th chord is long, then the double
integral corresponding to the $k$-th and $(k-1)$-th chords
is at most:

$$ \Big| \int_{t_{k-2}}^{t_c} \Big( \int_{t_{k-1}}^{t_c}
d(ln(z_k-z_k'))- \int_{t_{k-1}}^{t_c} d(ln(z_k''-z_k')) \Big)
d(ln(z_{k-1}-z_{k-1}')) \Big| \leq$$
$$\leq C \varepsilon \Big| \int_{t_{k-2}}^{t_c} d(ln(z_{k-1}-z_{k-1}')) \Big| =
C \varepsilon \Big| ln \frac{z_{k-1}(t_c)-z_{k-1}'(t_c)}{z_{k-1}(t_{k-2})
-z_{k-1}'(t_{k-2})}\Big|\leq$$
$$ \leq CC' \varepsilon,$$

\smallskip

where $C'$ is another constant depending on the ratio of the biggest and
smallest horizontal distance from the needle to the rest of the knot.

\smallskip

If the $(k-1)$-th chord is short, i.e. it connects $z_{k-1}$ and $z_{k-1}'$
that are both on the needle, then we can estimate the double integral
corresponding to the $k$-th and $(k-1)$-th chords:

$$ \Big| \int_{t_{k-2}}^{t_c} \Big( \int_{t_{k-1}}^{t_c} d(ln(z_k-z_k')) -
\int_{t_{k-1}}^{t_c} d(ln(z_k''-z_k')) \Big) \frac{dz_{k-1}'-dz_{k-1}}
{z_{k-1}'-z_{k-1}} \Big| \leq$$

$$\leq C \Big| \int_{t_{k-2}}^{t_c} (z_k''(t_{k-1})-z_k(t_{k-1}))
\frac{dz_{k-1}'-dz_{k-1}}{|z_{k-1}'-z_{k-1}|} \Big| =$$

$$=C \Big| \int_{t_{k-2}}^{t_c} d(z_{k-1}'-z_{k-1}) \Big|
= C|z_{k-1}'(t_{k-2})-z_{k-1}(t_{k-2})|
\leq C\varepsilon.$$

Continuing to go down the needle, we see that the difference between
$Z(T_0)$ and $Z(T_1)$ in degree $n$ is proportional to $(C'')^n \varepsilon$,
for a constant $C'' = max\{C,C'\}$, and by horizontal
deformations we can make $\varepsilon$ tend to zero, therefore the
difference tends to zero. This proves the proposition.

\end{proof}

This proves the invariance of the Kontsevich integral in the class of
Morse knots: To move critical points, one can form a sharp needle using
horizontal deformations only, then shorten or lengthen the needles
arbitrarily, then deform the knot as desired by horizontal deformations.

\bigskip

However, $Z$ is not invariant under ``straightening humps'', i. e.
deformations that change the number of critical points, as shown 
below. (We note that straightening the mirror image of the
hump shown is equivalent to this one, see Section \ref{corrections}
for the details.)

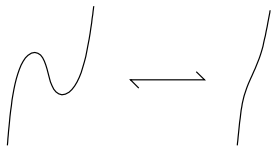
\begin{center}
\begin{picture}(0,0)%
\includegraphics{straighten.pstex}%
\end{picture}%
\setlength{\unitlength}{3947sp}%
\begingroup\makeatletter\ifx\SetFigFont\undefined%
\gdef\SetFigFont#1#2#3#4#5{%
  \reset@font\fontsize{#1}{#2pt}%
  \fontfamily{#3}\fontseries{#4}\fontshape{#5}%
  \selectfont}%
\fi\endgroup%
\begin{picture}(1674,899)(1189,-1323)
\end{picture}%

\end{center}

To fix this problem, we apply a correction, using the following proposition:

\begin{proposition} \label{hump}
Let $K$ and $K'$ be two knots differing only in a small hump in $K'$ that is
straightened in $K$ (as in the figure). Then
$$Z(K')=Z(K)Z(\raisebox{-2mm}{\begin{picture}(0,0)%
\includegraphics{humpcircle.pstex}%
\end{picture}%
\setlength{\unitlength}{3947sp}%
\begingroup\makeatletter\ifx\SetFigFont\undefined%
\gdef\SetFigFont#1#2#3#4#5{%
  \reset@font\fontsize{#1}{#2pt}%
  \fontfamily{#3}\fontseries{#4}\fontshape{#5}%
  \selectfont}%
\fi\endgroup%
\begin{picture}(325,321)(1115,-744)
\end{picture}%
}).$$
\end{proposition}

This proposition is a consequence of the following lemma:

\begin{lemma} \label{nointeraction}
{\bf Faraway strands don't interact.}
Let $K$ be a Morse knot with a distinguished tangle $T$, with $t_{bot}$
and $t_{top}$ being the minimal and maximal values of $t$ on $T$. Then,
in the formula of the Kontsevich integral, for those components whose
projection on the $t_j$ axis is contained in $[t_{bot},t_{top}]$,
it is enough to consider pairings where either both points $(z_j,t_j)$ and
$(z_j',t_j')$ belong to $T$, or neither do.
\end{lemma}

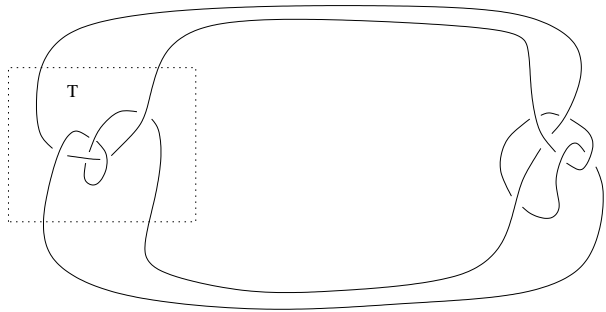
\begin{center}
\begin{picture}(0,0)%
\includegraphics{tink.pstex}%
\end{picture}%
\setlength{\unitlength}{2487sp}%
\begingroup\makeatletter\ifx\SetFigFont\undefined%
\gdef\SetFigFont#1#2#3#4#5{%
  \reset@font\fontsize{#1}{#2pt}%
  \fontfamily{#3}\fontseries{#4}\fontshape{#5}%
  \selectfont}%
\fi\endgroup%
\begin{picture}(5953,3054)(514,-2656)
\end{picture}%

\end{center}

\begin{proof}(Sketch)
We can shrink the tangle $T$ into a narrow box of width $\varepsilon$, and
do the same for the rest of the knot between heights $t_{bot}$ and $t_{top}$.
It's easy to see that the value of the integral corresponding to ``long''
chords (connecting the tangle to the rest of the knot) then tends to zero.
\end{proof}

Now let us sketch the proof of Proposition \ref{hump}:
\begin{proof}
The proposition follows by choosing $T$ to include just the hump, i.e.
there will be no long chords connecting the hump to the rest of the knot
in $K'$ or in \raisebox{-2mm}{\begin{picture}(0,0)%
\includegraphics{humpcircle.pstex}%
\end{picture}%
\setlength{\unitlength}{3947sp}%
\begingroup\makeatletter\ifx\SetFigFont\undefined%
\gdef\SetFigFont#1#2#3#4#5{%
  \reset@font\fontsize{#1}{#2pt}%
  \fontfamily{#3}\fontseries{#4}\fontshape{#5}%
  \selectfont}%
\fi\endgroup%
\begin{picture}(325,321)(1115,-744)
\end{picture}%
}.

Also, there cannot be any chords above or below the hump, since the highest
(resp. lowest) of those would be an isolated chord.
\end{proof}

Detailed proofs of the previous two statements can be found in \cite{chmutov}

Since the constant term of $Z(\raisebox{-2mm}{\begin{picture}(0,0)%
\includegraphics{humpcircle.pstex}%
\end{picture}%
\setlength{\unitlength}{3947sp}%
\begingroup\makeatletter\ifx\SetFigFont\undefined%
\gdef\SetFigFont#1#2#3#4#5{%
  \reset@font\fontsize{#1}{#2pt}%
  \fontfamily{#3}\fontseries{#4}\fontshape{#5}%
  \selectfont}%
\fi\endgroup%
\begin{picture}(325,321)(1115,-744)
\end{picture}%
})$ is $1$, it has a
reciprocal in the graded completion
of ${\mathcal A}$ (i.e. formal infinite series of chord diagrams). Using
this we can now define an honest knot invariant $Z'$ by setting
$$Z'(K)=\frac{Z(K)}{Z(\raisebox{-2mm}{\begin{picture}(0,0)%
\includegraphics{humpcircle.pstex}%
\end{picture}%
\setlength{\unitlength}{3947sp}%
\begingroup\makeatletter\ifx\SetFigFont\undefined%
\gdef\SetFigFont#1#2#3#4#5{%
  \reset@font\fontsize{#1}{#2pt}%
  \fontfamily{#3}\fontseries{#4}\fontshape{#5}%
  \selectfont}%
\fi\endgroup%
\begin{picture}(325,321)(1115,-744)
\end{picture}%
})^{c/2}},$$ where $c$ is
the number of
critical points in the Morse embedding of $K$ that we use to compute $Z$.

\subsection{Universality}
Here we state Kontsevich's theorem, and the main idea of the proof, which will apply in the 
case of the extension to graphs word by word. A complete, detailed proof can 
be found in \cite{chmutov} or \cite{dbn}, and Kontsevich's paper \cite{konts}.

\begin{theorem}
Let $w$ be a weight system of order $n$. Then there exists a Vassiliev invariant of
order $\leq n$ whose weight system is $w$, given by the formula
$$K \mapsto w(Z'(K)).$$
\end{theorem}

This property of $Z'$ is referred to as being a {\it universal finite type invariant}.

\begin{proof}(Sketch.) 
Let $D$ be a chord diagram of order $n$, and $K_D$ a singular knot
with chord diagram $D$. The theorem follows from the fact that
$$Z'(K_D)=D+\{{\rm higher \thinspace order \thinspace terms}\}.$$
Since the denominator of $Z'$ always begins with $1$ (the unit of ${\mathcal A}$),
it is enough to prove that
$$Z(K_D)=D+\{{\rm higher \thinspace order \thinspace terms}\}.$$
Because of the factorization property and the fact that faraway strands don't
interact (Lemma \ref{nointeraction}), we can think locally. Around a 
single double point, we need to compute the difference of $Z$ on an over-crossing
and an under-crossing. These can be deformed as follows:
$$ Z(\overcrossing)-Z(\undercrossing)=
Z \big(\raisebox{-0.08 in}{\begin{picture}(0,0)%
\includegraphics{over.pstex}%
\end{picture}%
\setlength{\unitlength}{1342sp}%
\begingroup\makeatletter\ifx\SetFigFont\undefined%
\gdef\SetFigFont#1#2#3#4#5{%
  \reset@font\fontsize{#1}{#2pt}%
  \fontfamily{#3}\fontseries{#4}\fontshape{#5}%
  \selectfont}%
\fi\endgroup%
\begin{picture}(611,894)(1474,-2488)
\end{picture}%
 } \big)-
Z \big(\raisebox{-0.08 in}{\begin{picture}(0,0)%
\includegraphics{under.pstex}%
\end{picture}%
\setlength{\unitlength}{1342sp}%
\begingroup\makeatletter\ifx\SetFigFont\undefined%
\gdef\SetFigFont#1#2#3#4#5{%
  \reset@font\fontsize{#1}{#2pt}%
  \fontfamily{#3}\fontseries{#4}\fontshape{#5}%
  \selectfont}%
\fi\endgroup%
\begin{picture}(549,848)(2014,-2473)
\end{picture}%
 } \big).$$
Since the crossings on the bottom are now identical, by the factorization
property, it's enough to consider 
$$ Z \big( \raisebox{-0.06 in}{\begin{picture}(0,0)%
\includegraphics{loop.pstex}%
\end{picture}%
\setlength{\unitlength}{1342sp}%
\begingroup\makeatletter\ifx\SetFigFont\undefined%
\gdef\SetFigFont#1#2#3#4#5{%
  \reset@font\fontsize{#1}{#2pt}%
  \fontfamily{#3}\fontseries{#4}\fontshape{#5}%
  \selectfont}%
\fi\endgroup%
\begin{picture}(611,734)(1474,-2328)
\end{picture}%
 } \big)-
Z \big( \raisebox{-0.06 in}{\begin{picture}(0,0)%
\includegraphics{straight.pstex}%
\end{picture}%
\setlength{\unitlength}{1342sp}%
\begingroup\makeatletter\ifx\SetFigFont\undefined%
\gdef\SetFigFont#1#2#3#4#5{%
  \reset@font\fontsize{#1}{#2pt}%
  \fontfamily{#3}\fontseries{#4}\fontshape{#5}%
  \selectfont}%
\fi\endgroup%
\begin{picture}(549,698)(2014,-2323)
\end{picture}%
 } \big).$$

$ Z \big( \raisebox{-0.06 in}{\begin{picture}(0,0)%
\includegraphics{straight.pstex}%
\end{picture}%
\setlength{\unitlength}{1342sp}%
\begingroup\makeatletter\ifx\SetFigFont\undefined%
\gdef\SetFigFont#1#2#3#4#5{%
  \reset@font\fontsize{#1}{#2pt}%
  \fontfamily{#3}\fontseries{#4}\fontshape{#5}%
  \selectfont}%
\fi\endgroup%
\begin{picture}(549,698)(2014,-2323)
\end{picture}%
 } \big)$ equals $1$ 
(the unit of ${\mathcal A}(\uparrow_2)$, where ${\mathcal A}(\uparrow_2)$
stands for chord diagrams on two upward oriented vertical strands), as 
both $z_i(t)$ and $z_i'(t)$ are constant.

In $Z \big( \raisebox{-0.06 in}{\begin{picture}(0,0)%
\includegraphics{loop.pstex}%
\end{picture}%
\setlength{\unitlength}{1342sp}%
\begingroup\makeatletter\ifx\SetFigFont\undefined%
\gdef\SetFigFont#1#2#3#4#5{%
  \reset@font\fontsize{#1}{#2pt}%
  \fontfamily{#3}\fontseries{#4}\fontshape{#5}%
  \selectfont}%
\fi\endgroup%
\begin{picture}(611,734)(1474,-2328)
\end{picture}%
 } \big)$, the first
term is $1$, as always, so this will cancel out in the difference.
The next term is the chord diagram with one single chord, and this
has coefficient $\frac{1}{2 \pi i} \int_{t_{min}}^{t_{max}} 
\frac{dz-dz'}{z-z'} =1,$ by Cauchy's theorem. 

So the lowest degree term of the difference is a single chord with
coefficient one.

Now putting $K_D$ together, the lowest degree term in $Z(K_D)$ will
be a chord diagram that has a single chord for each double point,
which is exactly $D$.

\end{proof}

\section{The naive extension}

\subsection{Knotted trivalent graphs}
Let us first state the necessary definitions and basic properties.

\smallskip

A {\it trivalent graph} is a graph which has three edges meeting at each
vertex. We will require that all edges be oriented. We allow multiple edges;
loops (i.e. edges that begin and end at the same vertex); and circles (i.e.
edges without a vertex).

\smallskip

A {\it knotted trivalent graph (KTG)} is an isotopy class of embeddings of a
fixed such graph in $\mathbb{R}^3$. (So in particular, knots and links
are knotted trivalent graphs.)
We will also require edges to be equipped with a framing, i. e.
a choice of a normal vector field, which is smooth along
the edges and is chosen so that the three normal vectors agree at the
vertices. We can imagine the graph being a
``band graph'', with thickened edges, as shown in the picture.

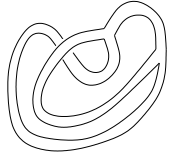
\begin{center}
\begin{picture}(0,0)%
\includegraphics{band.pstex}%
\end{picture}%
\setlength{\unitlength}{2131sp}%
\begingroup\makeatletter\ifx\SetFigFont\undefined%
\gdef\SetFigFont#1#2#3#4#5{%
  \reset@font\fontsize{#1}{#2pt}%
  \fontfamily{#3}\fontseries{#4}\fontshape{#5}%
  \selectfont}%
\fi\endgroup%
\begin{picture}(1866,1766)(1372,-2534)
\end{picture}%

\end{center}

Isotopy classes of KTG's are in one to one correspondence
with graph diagrams (projections onto a plane with only transverse double
points preserving the over- and under-strand information at the crossings),
modulo the Reidemeister moves $R2$, $R3$ and $R4$ (first defined for
graphs in \cite{yamada}). $R1$ is omitted because
we're working with framed graphs, $R2$ and $R3$ are the same as in the knot
case. $R4$ involves moving a strand in front of or behind a vertex:

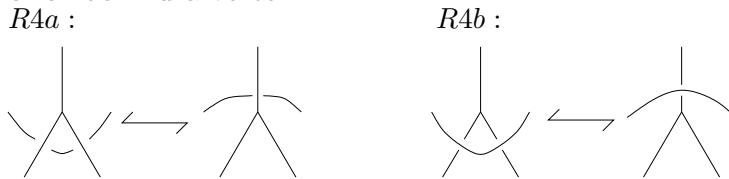
\begin{center}
\input r4.pstex_t
\end{center}

\smallskip
{\it Trivalent tangles} are defined the same way as tangles, with the difference
that we allow trivalent vertices inside the cube. In other words, a trivalent
tangle is a uni-trivalent graph embedded in a cube, such that the positions of
the univalent vertices are fixed on the bottom and top faces of the cube, as before.
Here we require the edges to be oriented and framed, and regard trivalent 
tangles up to boundary-fixing ambient isotopy.

Trivalent tangles with the appropriate numbers of edge endings can be
multiplied the same way as ordinary tangles, by stacking the cubes and
rescaling.

\smallskip
There are four operations defined on KTG's:

\smallskip

Given a trivalent graph (knotted or not) $\Gamma$ and an edge $e$ of 
$\Gamma$, we can {\it switch the orientation} of $e$. We denote the
resulting graph by $S_e(\Gamma)$.

We can also {\it delete} the edge
$e$, which means the two vertices at the ends of $e$ also cease to exist to
preserve the trivalence. To do this, it is required that the orientations of the
two edges connecting to $e$ at either end match. The resulting graph will be
denoted by $d_e(\Gamma)$.

\smallskip

{\it Unzipping} the edge $e$ (denoted $u_e(\gamma)$ for a KTG $\gamma$, 
see figure below)
means replacing it by two edges that are very
close to each other - to do this we use the framing, and 
thus unzip is only well defined on a KTG. The
two vertices at the ends of $e$ will disappear. (It can be imagined as cutting
the band of $e$ in half lengthwise.) Again, the orientations have to match,
i.e. the edges at the vertex where $e$ begins have to both be incoming,
while the edges at the vertex where $e$ ends must both be outgoing.

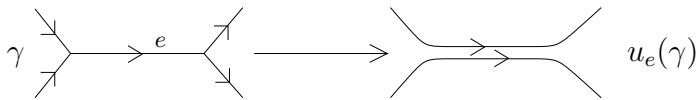
\begin{center}
\input unzip.pstex_t
\end{center}

We would like to stress here that the unzip of an edge of an unknotted
trivalent graph $\Gamma$ is not well-defined: there are two possible
ways of connecting the two edges on the left with the two edges on the
right. In the future, when we write ``$u_e(\Gamma)$'', we will be refering 
to the combinatorial object (skeleton) behind the well-defined knotted
trivalent graph $u_e(\gamma)$.

Another approach to handle this inconvenience is to have trivalent
graphs (knotted or not) come equipped with a vertex orientation, i.e. a
cyclic ordering of the edges at each vertex, which makes the unzip
well-defined. \cite{murakami}, for example, uses this version of the 
definition of KTG's.

\smallskip

Given two graphs with selected edges $(\Gamma,e)$ and $(\Gamma', f)$,
the {\it connected sum} of these graphs along the two chosen edges,
denoted $\Gamma \#_{e,f} \Gamma'$, is
obtained by joining $e$ and $f$ by a new edge. For this to be well-defined,
we also need to specify the direction of the new edge, the framing on it,
and, thinking of $e$ and $f$ as bands, which side of the bands the new
edge is attached to. To compress notation, let us declare that the new edge
be oriented from $\Gamma$ towards $\Gamma'$, have no twists,
and, using the blackboard framing, be attached to the right side
of $e$ and $f$, as shown:

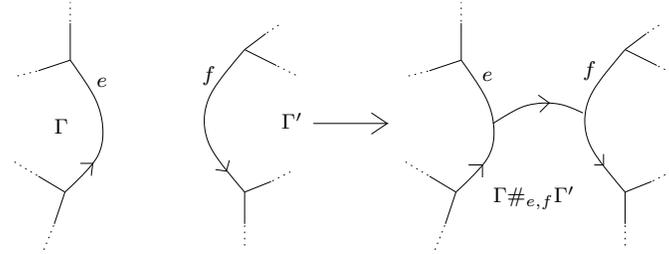
\begin{center}
\input connsum.pstex_t
\end{center}

\subsection{The algebra ${\mathcal A}(\Gamma)$}
The extended integral will take its values in the algebra
$\mathcal{A}(\Gamma)$,
which consists of chord diagrams on the {\it skeleton} $\Gamma$, the trivalent
graph (as a combinatorial object, not as embedded in $\mathbb{R}^3$), and
is again factored out by the same four term relations, along with one
more class of relations called the vertex invariance, or $VI$, relations
(branching relation in \cite{murakami}):

\begin{center}
\input vi.pstex_t
\end{center}

Here, the sign $(-1)^{\rightarrow}$ is $-1$ if the edge the chord is ending
on is oriented to be outgoing from the vertex, and $+1$ if it's incoming.

\bigskip

The vertex invariance relation arises, topologically, the same way as the
$4T$ relation, i.e. any weight system of a finite type invariant of knotted
trivalent graphs (defined the same way as for knots) will be zero on the
above sum of chord diagrams. To prove this we can either use the same isotopy 
that was used for the $4T$ relation, or Reidemeister 4.

\smallskip

We are going to study invariants of framed graphs, hence we will not
factor out by the one-term (or framing independence) relation.

There are several reasons to do
this, one being that the unzip operation uses the framing, so it's
natural to want invariants to be perceptive of it as well. Also, while
up to isotopy, all the information in a framing of a knot can be described by
an integer (the number of twists), this is no longer true for graphs.
A ``local'' twist of an edge can, through an isotopy, become a ``global''
twist of two other edges connecting to it at a vertex. Therefore, one can argue
that framing is a more interesting property of graphs than of knots.
Motivation for studying a framed version of the Kontsevich integral for
links arises from the relationship to 3-manifolds, through surgery 
(see \cite{dbn}; \cite{lmo}; \cite{ohts}; \cite{wat}).

Now let us define operations on the spaces ${\mathcal A}(\Gamma)$:

\smallskip

Given a graph $\Gamma$ and an edge $e$, the orientation switch
operation is a linear map $S_e:{\mathcal A}(\Gamma) \to {\mathcal A}(s_e(\Gamma))$
that multiplies a chord diagram $D$ by $(-1)^k$ where $k$ is the number of chords
in $D$ ending on $e$. This generalizes the antipode map on Jacobi diagrams, which 
corresponds to the orientation reversal of knots (see \cite{ohts}, p.136).

\smallskip

Edge delete is a linear map
$d_e:{\mathcal A}(\Gamma) \to {\mathcal A}(d_e(\Gamma))$, defined as follows:
when the edge $e$ is deleted, all diagrams that had a chord ending on $e$ will
become zero, with all other chords unchanged. Edge delete is the generalization
of the co-unit map of \cite{ohts} (p.136), and \cite{dbn}.

\smallskip
There is an operation on ${\mathcal A}(O)$ corresponding to the cabling of knots:
references include \cite{dbn} (splitting map) and \cite{ohts} (co-multiplication). 
The graph unzip operation is
the graph analogy of cabling, so the corresponding map is analogous as well:

Unzip is a linear map
$u_e:{\mathcal A}(\Gamma) \to {\mathcal A}(u_e(\Gamma)).$
When $e$ is unzipped, each chord that ends on it will be
replaced by a sum of two chords, one ending on each new edge (i.e., if
$k$ chords end on $e$, then $u_e$ will send this particular chord diagram
to a sum of $2^k$ chord diagrams). 

\smallskip

For graphs $\Gamma$ and $\Gamma'$, with edges $e$ and $e'$,
the connected sum
$\#_{e,e'}:{\mathcal A}(\Gamma) \times {\mathcal A}(\Gamma')
\to {\mathcal A}(\Gamma \#_{e,e'} \Gamma')$
is defined in the obvious way, by performing the connect sum
operation on the skeletons and not changing the chords in any way. This is
well defined due to the $4T$ and $VI$ relation. (What needs to be proved
is that we can move a chord ending over the attaching point of the new edge;
this is done in the same spirit as the proof of Lemma 3.1 in \cite{dbn}, 
using ``hooks''; see also \cite{murakami}, figure 4.)

\smallskip

It is easy to check that all the operations are well-defined (agree with
the 4T and VI relations).

\subsection{The naive Kontsevich integral of a graph}
We can try to extend the definition to knotted trivalent graphs
(and trivalent tangles) the natural
way: consider a Morse embedding of the graph (or tangle) in
$\mathbb{R}^3$, and define
the integral by the same formula, requiring that $t_1,...,t_n$ are
non-critical and also not the heights of vertices. (We do not do any
correction or renormalization yet.)

\subsection{The good properties} The ordinary Kontsevich integral 
has several nice properties which, we can realistically hope, should
hold true for the extension and corresponding graph operations: 

\begin{itemize}
\item {\bf Factorization property.}

The extension of the integral
would obviously preserve the factorization property
of the Kontsevich integral, meaning that it would be multiplicative with
respect to stacking trivalent tangles and the (vertical) connected
sum of graphs (i.e. it would commute with the connected sum operation).

\smallskip

\item {\bf Nice behavior under orientation switches.}

$Z$ will commute with the orientation switch operation due to
the signs $(-1)^{\downarrow}$ in the formula that defines $Z$.
(In other words, when we switch the orientation of an edge, the
coefficients of each chord diagram in the result of the integral
will be multiplied by $(-1)^k$, where $k$ is the number of
chords ending on the edge we switched the orientation of).

\smallskip

\item {\bf Nice behavior under the vertical edge delete operation.}

Let us (wrongly) assume that our extension is a convergent knotted graph
invariant. Consider an embedding of the graph in which the edge
$e$ to delete is a straight vertical line, with the top
vertex forming a $Y$, the bottom vertex a $\lambda$.
(Obviously such an embedding exists within each isotopy class.)

Now, if we delete the edge $e$, then in the result of the integral, every chord
diagram in which a chord ended on $e$ would disappear (declare these to be
zero), and the coefficient of any other chord diagram stays unchanged (as
the integral used to compute it is unchanged). In other words, the extended
Kontsevich integral commutes with the edge delete operation.

\smallskip

\item {\bf Nice behavior under vertical unzip.}

Let the embedding of the graph be as above.
When we unzip the vertical edge $e$, we do it so that the two new edges
are parallel and very close to each other.

In the result of the integral, the chord diagrams that contained $k$
chords ending on $e$ will be replaced by a sum of $2^k$ chord diagrams,
as each chord is replaced by ``the sum of two chords'', one of them ending
on the first new edge, the other ending on the second. (Since for each choice
of $z_i$ on $e$ we will now have two choices.) The coefficient for
the sum of these new diagrams will be the same as the coefficient of their
``parent'', (since the two new edges are arbitrarily close to each other).

If we were to choose a chord to have both ends on the two new parallel
edges, the resulting integral will be zero, as $z_i-z_i'$ will be a constant
function.

Again, the coefficients of the diagrams that don't involve chords ending on
$e$ are unchanged.
Therefore, the extended Kontsevich integral, assuming it exists and is an
invariant, will commute with the unzip operation.
\end{itemize}

\subsection{The problem}

The problem with the extension is that the integral, as
defined above, is divergent. Causing this are the possible short
chords near a vertex (i.e. those not separated from a vertex by another chord
ending).
These are just like the isolated chords in the knot case,
but, contrary to the knot case, we have no reason to factor out by all the
chord diagrams containing such chords.

Also, if we want to drop the 1T relation for the sake of working with
framed graphs, we have to fix the divergence coming from the isolated
chords near critical points as well.

\section{Eliminating the divergence}

To eliminate the divergence we have to re-normalize at the vertices, and
critical point. We
do this using a simple version of a renormalization technique from quantum 
field theory: we know the exact type of divergence, and thus we ``divide 
by it'' to get a convergent integral.

\subsection{The re-normalized integral $Z_2$}
First let us restrict our attention to a vertex of a ``$\lambda$'' shape.
Fix a scale $\mu$ and chose a small $\varepsilon$. We change the
integral at the vertex by ``opening up'' the two lower strands at
a distance $\varepsilon$ from the vertex, to a width $\mu$ at the
height of the vertex.

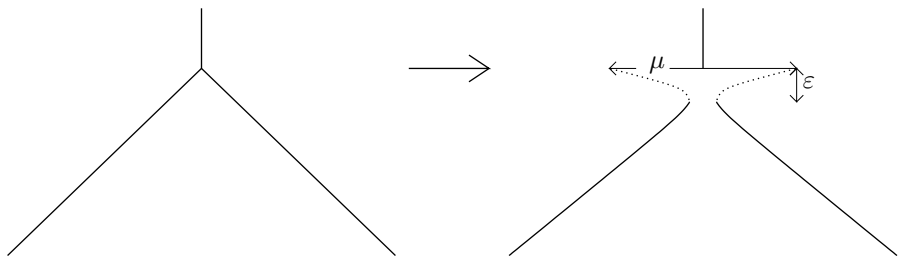
\begin{center}
\input renormalize.pstex_t
\end{center}

The old strands (solid lines on the picture) up to distance $\varepsilon$
from the vertex, and above the vertex, will be ``globally active'',
meaning that we allow any chords (long or short) to end on them.
The opening strands (dashed lines on the picture) are ``locally
active'', meaning that we allow chords between them, but chords
from outside are not allowed to end on them. We define the value
of $Z_2$ as the limit of this new integral as $\varepsilon$ tends
to zero.

\smallskip

We will do the same to a vertex of a ``$Y$''-shape, however,
we will have to restrict our attention to these two types of vertices.
(I.e. we do not allow vertices to be local minima or maxima.)
Of course, any graph can be embedded in ${\mathbb R}^3$ in such a way
that all vertices are of one of these two types, but this will cause
a problem with the invariance of $Z_2$,
which will need to be fixed.

\smallskip

To get an invariant of framed graphs, we use the same method to
re-normalize at the critical points and thereby make isolated chords
cause no divergence, this is why we can drop the one term (framing
independence) relation.

\begin{proposition} \label{graphconvergence}
The re-normalized integral $Z_2$ is convergent.
\end{proposition}

\begin{proof}
It suffices to consider the case of a $\lambda$-shaped vertex, the other 
cases being strictly similar. Let us fix a $\lambda$-shaped vertex $v$.
The globally active part corresponding to the highest short chord
$c_i$ can be computed as below. By Lemma \ref{nointeraction}, we do
not need to consider long chords ending on the globally active part,
so we only need to deal with short chords.

$$\int_{t_{i-1}}^{t_v-\varepsilon}\frac{dz_i-dz_i'}{z_i-z_i'} =
\int_{t_{i-1}}^{t_v-\varepsilon}d ln(z_i-z_i') =
ln \Big( \frac{z_i(t_v-\varepsilon)-z_i'(t_v-\varepsilon)}
{z_i(t_{i-1})-z_i'(t_{i-1})} \Big).$$

The locally active part on the other hand:
$$\int_{t_v-\varepsilon}^{t_v} d(ln(z_i-z_i')) =
ln \Big(\frac{\mu}{z_i(t_v-\varepsilon)-z_i'(t_v-\varepsilon)} \Big).$$

The integral for this highest chord is the sum of these, and is
therefore equal to:
$$ln \Big( \frac{\mu}{z_i(t_{i-1})-z_i'(t_{i-1})} \Big) = 
ln \mu - ln(z_i(t_{i-1})-z_i'(t_{i-1})) .$$

If there is another short chord underneath, $c_{i-1}$, then
$z_i(t_{i-1})=z_{i-1}$, and $z_i'(t_{i-1})=z'_{i-1}$
(or vice versa, which does not matter in the value of $Z$).
So, proceeding in a similar fashion as above, the double 
integral corresponding to the two chords is:

$$\int_{t_{i-2}}^{t_v} \Big( ln \mu - 
ln (z_{i-1}-z'_{i-1}) \Big) dln (z_{i-1}-z'_{i-1}) =$$

$$=\Big[- \frac{1}{2} \big(ln \mu - ln (z_{i-1}-z'_{i-1})\big)^2 \Big]_{t_{i-2}}^{t_v}=
\frac{1}{2} \Big( ln \mu -ln \big(z_{i-1}(t_{i-2})-z'_{i-1}(t_{i-2})\big) \Big)^2. $$

We continue in this fashion for as many short chords as
there are between $t_v$ and the next critical or vertex level below,
let us denote this level by $t_c$. For $k$ short chords, we see that the value
of the integral between the two critical levels is:

$$\frac{1}{k!} \Big( ln \mu -ln\big(z_{i-k+1}(t_c)-z'_{i-k+1}(t_c)\big) \Big)^k.$$

If at the level $t_c$, the critical point or vertex involves the
same two strands, then $z_{i-k+1}(t_c)-z'_{i-k+1}(t_c)= \mu$,
due to the renormalization at that critical point or vertex, so the
result is $\frac{1}{k!}(ln \mu -ln \mu)=0$. If it does not involve the same two 
strands, then 
$$\big| \frac{1}{k!} (ln \mu -ln\big(z_{i-k+1}(t_c)-z'_{i-k+1}(t_c)\big) \Big| \leq C,$$
where C is a constant that depends on the distance between the two strands at 
the level $t_c$, which is independent of $\varepsilon$.

Therefore, the integral is convergent, and remains convergent as we take 
$\varepsilon \rightarrow 0$.

\end{proof}

\subsection{The good properties}

\smallskip

Let us call the deletion (respectively, unzip) of an edge that is embedded as
a vertical line segment {\it vertical edge delete} (respectively, {\it vertical 
unzip}). By {\it vertical connected sum}, we mean placing one KTG above the another and
connecting them by an edge that is a vertical line segment.

\begin{theorem} \label{goodprop}
$Z_2$
is invariant under horizontal deformations that leave the critical points
and vertices fixed,
and rigid motions of the critical points
and vertices. $Z_2$ has the factorization
property, and commutes with orientation switch, vertical edge delete, 
edge unzip and connected sum.
Moreover, it has good behavior under changing the renormalization
scale $\mu$.

\end{theorem}

By rigid motions of critical points we mean shrinking or extending a sharp needle,
like in the case of the standard Kontsevich integral (Lemma \ref{mvcrits}),
with the difference that
we do not allow twists on the needle, but require the two sides of the needle to
be parallel straight lines. This difference is due to dropping the
framing independence relation, as adding or eliminating twists would change the
framing.

For vertices, a rigid motion is moving the vertex
down two very close edges without twists, as shown in the figure:

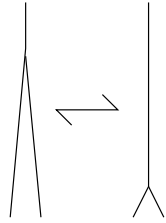
\begin{center}
\begin{picture}(0,0)%
\includegraphics{vertexmotion.pstex}%
\end{picture}%
\setlength{\unitlength}{3947sp}%
\begingroup\makeatletter\ifx\SetFigFont\undefined%
\gdef\SetFigFont#1#2#3#4#5{%
  \reset@font\fontsize{#1}{#2pt}%
  \fontfamily{#3}\fontseries{#4}\fontshape{#5}%
  \selectfont}%
\fi\endgroup%
\begin{picture}(989,1374)(1339,-973)
\end{picture}%

\end{center}

To prove that the integral commutes with the vertical edge unzip operation and
to investigate the behavior under changing the scale $\mu$, we will use
the following lemma:

\begin{lemma} \label{rescale}
Let $w_1, w_2$ be distinct complex numbers and let $\beta$ be
another complex number. Let $B$ be the 2-strand ``rescaling braid'' defined
by the map
\begin{center}

$[\tau,T] \to [\tau,T]\times{\mathbb C}^2$

$t \mapsto (t, e^{\beta t}w_1,e^{\beta t}w_2).$
\end{center}

Then
$$Z_2(B)=exp \Big(\frac{\beta t_{12}(T-\tau)}{2 \pi i}\Big) \in {\mathcal A}(\uparrow_2),$$
where $t_{12}$ is the chord diagram with one
chord between the two vertical strands.
\end{lemma}

\begin{proof}
The $m$-th term of the sum in the defining formula of $Z$ is
$$\frac{1}{(2\pi i)^m} t_{12}^m \int_{\tau}^{T}\int_{t_1}^{T}...\int_{t_{m-1}}^{T}
d ln(e^{\beta t_m}w_1-e^{\beta t_m}w_2)...d ln(e^{\beta t_1}w_1-e^{\beta t_1}w_2)=$$
$$=\frac{1}{(2\pi i)^m} t_{12} \beta^m \int_{\tau}^{T}\int_{t_1}^{T}...\int_{t_{m-1}}^{T}
dt_m dt_{m-1}...dt_1
=\frac{(\beta t_{12} (T-\tau))^m}{(2\pi i)^m m!},$$
which proves the claim.

\end{proof}

We note that Lemma \ref{rescale} is easily extended to the case of the $n$-strand
rescaling braid, defined the same way, where in the statement $t_{12}$ would be
replaced by $\sum t_{ij}$.

We also state the following reformulation, that follows from Lemma \ref{rescale} by
elementary algebra:

\begin{lemma}\label{rescale2}
For the two strand rescaling braid $B$ where the bottom distance between the
strands is $l$, and the top distance is $L$, such as this one: 
\raisebox{-0.1 in}{\input twobraid.pstex_t} ,
$$Z_2(B)= exp \Big( \frac{ln(L/l) t_{12}}{2 \pi i}\Big),$$
independently of $T$ and $\tau$.
\end{lemma}

Now we proceed to prove Theorem \ref{goodprop}:

\begin{proof}

{\bf Factorization property.}

The factorization property for tangles is
untouched by the renormalization, as the height at which tangles are
glued together must be non-critical and not contain any vertices.

For the vertical connected sum of knotted graphs $\gamma_1$ and $\gamma_2$,
denoted $\gamma_1 \# \gamma_2$,
if we connect the maximum point of the $\gamma_1$ with the minimum of
$\gamma_2$, the minimum and maximum renormalizations will become vertex
renormalizations when computing
the Kontsevich integral of $\gamma_1 \# \gamma_2$.

{\bf Invariance.}

To prove invariance under horizontal deformations that leave the critical points
and vertices fixed, we use the same proof as in the case of the standard
integral (Proposition \ref{hor}), i.e. cut the graph into tangles with no critical points or vertices,
and thin tangles containing the vertices and critical points, apply Lemma \ref{braidhor}
to the former kind, then take a limit.

\smallskip

For invariance under rigid motions of critical points, it is enough to 
consider the case of a maximum, the case of a minimum being strictly similar. 
Since we have proven
the invariance under horizontal deformations and the needle is not twisted,
we can assume that the sides of the needle are two parallel lines and $\varepsilon$
is the horizontal distance between them.

By the factorization property, the value of $Z_2$ for the needle extended 
(see figure) can
be written as a product of the values for the part under the needle, the two
parallel strands, and the renormalization for the critical point that is the
tip of the needle:

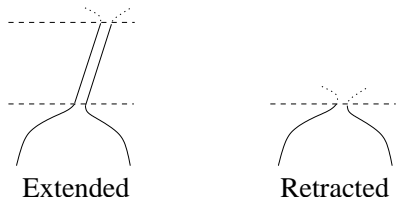
\begin{center}
\begin{picture}(0,0)%
\includegraphics{needle2.pstex}%
\end{picture}%
\setlength{\unitlength}{2131sp}%
\begingroup\makeatletter\ifx\SetFigFont\undefined%
\gdef\SetFigFont#1#2#3#4#5{%
  \reset@font\fontsize{#1}{#2pt}%
  \fontfamily{#3}\fontseries{#4}\fontshape{#5}%
  \selectfont}%
\fi\endgroup%
\begin{picture}(4529,2217)(124,-196)
\end{picture}%

\end{center}

The value of $Z_2$ for the needle retracted is the product of the value for
the part under the needle and the renormalization part. What we have to show therefore
is that in the first case (needle extended) the coefficients for any diagram
that contains any chords on the parallel strands tends to zero as
the width of the needle tends to zero.

This is indeed the case: the integral is $0$ for any diagram on two
parallel strands that contains any short chord, since $d(z_k-z_k')=0$.
For long chords, the highest long chord can be paired up with the one
ending on the other strand, as in the proof of \ref{mvcrits}. The reason
for this is that their difference commutes with any short chords that
occur in the renormalization part, by the Locality Lemma \ref{loc}. 
Now we can use the same estimates as in Proposition \ref{mvcrits} to finish the proof.

\smallskip

To prove invariance under rigid motions of vertices, let us assume that all
edges are outgoing. All other cases are proven the same way after inserting
the appropriate sign changes. Similarly to the
case of critical points, we can assume that the part we shrink consists of two parallel
strands at horizontal distance $\varepsilon$. We need to prove that the
difference of the values of $Z_2$ for the two pictures shown below tends to
zero as $\varepsilon$ tends to zero.

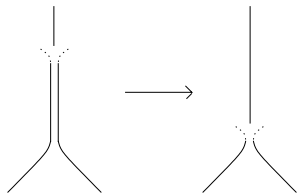
\begin{center}
\begin{picture}(0,0)%
\includegraphics{rigmotion.pstex}%
\end{picture}%
\setlength{\unitlength}{2171sp}%
\begingroup\makeatletter\ifx\SetFigFont\undefined%
\gdef\SetFigFont#1#2#3#4#5{%
  \reset@font\fontsize{#1}{#2pt}%
  \fontfamily{#3}\fontseries{#4}\fontshape{#5}%
  \selectfont}%
\fi\endgroup%
\begin{picture}(3323,2154)(574,-1108)
\end{picture}%

\end{center}

For the value corresponding to the left picture, just like in the needle
case, we can assume that there are no short chords
connecting the two parallel strands. The long chords ending on the
parallel strands come in pairs, with the same sign, and their coefficients
are the same in the limit. These pairs commute with any short chords in the
renormalization part by Lemma \ref{loc}. Also, by the vertex invariance 
relation,
each sum of a pair of such chords equals one chord ending on the top vertical edge,
which, from the right side integral, will have the same coefficient as the
former sum, as $\varepsilon \to 0$. This concludes the proof.

\smallskip

{\bf Good behavior under orientation switch}

The renormalization doesn't change anything about the signs
that correspond to the orientations of the edges, so $Z_2$
still commutes with orientation switches.

{\bf Good behavior under vertical edge delete.}

When deleting a vertical edge, the renormalization that was
originally inserted for the two vertices at either end of the
edge becomes exactly the
renormalization we need for the two critical points that replace the
vertices, as on the figure.

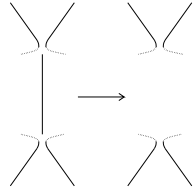
\begin{center}
\begin{picture}(0,0)%
\includegraphics{delete.pstex}%
\end{picture}%
\setlength{\unitlength}{710sp}%
\begingroup\makeatletter\ifx\SetFigFont\undefined%
\gdef\SetFigFont#1#2#3#4#5{%
  \reset@font\fontsize{#1}{#2pt}%
  \fontfamily{#3}\fontseries{#4}\fontshape{#5}%
  \selectfont}%
\fi\endgroup%
\begin{picture}(6376,6440)(726,-6026)
\end{picture}%

\end{center}

{\bf Good behavior under vertical edge unzip.}

It is slightly harder to see that $Z_2$ commutes with unzipping
an edge. If we unzip a vertical edge $e$ and then compute the
value of $Z_2$, the vertices on either end of $e$ disappear,
so no renormalization will occur.

However, if we first compute the integral and then perform unzip on the
result in ${\mathcal A}(\Gamma)$, then the coefficients for each resulting chord
diagram will be as if we had computed them using re-normalizations as in
the picture below.

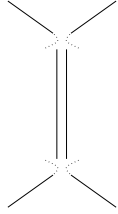
\begin{center}
\begin{picture}(0,0)%
\includegraphics{renormunzip.pstex}%
\end{picture}%
\setlength{\unitlength}{1066sp}%
\begingroup\makeatletter\ifx\SetFigFont\undefined%
\gdef\SetFigFont#1#2#3#4#5{%
  \reset@font\fontsize{#1}{#2pt}%
  \fontfamily{#3}\fontseries{#4}\fontshape{#5}%
  \selectfont}%
\fi\endgroup%
\begin{picture}(2562,4845)(940,-4359)
\end{picture}%

\end{center}

What we need to show is that the contribution from the ``upper''
renormalization will cancel the contribution from the ``lower''.

Let $T$ denote the trivalent tangle from the lower
renormalization to the upper renormalization (including any ``faraway
part'' that is not on the picture).
Let us divide $T$ into three tangles:
let $T_1$ denote the lower renormalization area (short chords only),
$T_2$ the unzipped edge and any faraway part of the graph at
this height, and $T_3$ the upper renormalization area.

By the factorization property, $Z(T)=Z(T_1)Z(T_2)Z(T_3)$

In the integral's result, the short chords occurring at either renormalization
part can slide up and down the unzipped edge by Lemma \ref{loc}, as they
commute with the pairs of incoming long chords.

In other words, $Z(T_1)$ commutes with $Z(T_2)$, and therefore
$Z(T)=Z(T_2)Z(T_1)Z(T_3)$.

$Z(T_1)Z(T_3)=exp \Big( \frac {ln(\mu / \varepsilon) t_{12}} {2\pi i} \Big)
exp \Big( \frac {ln(\varepsilon/\mu) t_{12}}{2\pi i}\Big)=
1 \in {\mathcal A}(T)$ by the reformulated rescaling Lemma \ref{rescale2}, so the
re-normalizations cancel each other.

{\bf Good behavior under changing the scale $\mu$.}

Changing the scale $\mu$ to some other $\mu'$ amounts to adding a small, 
two-strand ``rescaling braid'' at each vertex and critical point:

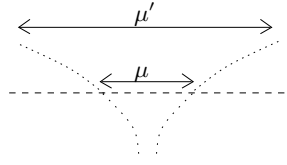
\begin{center}
\input rescale.pstex_t
\end{center}

By Lemma \ref{rescale2}, this means that when changing the
scale from $\mu$ to $\mu'$, the element
$exp \Big( \frac {ln(\mu'/\mu) t_{12}}{2\pi i} \Big) \in \mathcal{A}(\uparrow _2)$
is placed at each $\lambda$-shaped vertex and maximum point, and
$exp \Big( \frac {ln(\mu/\mu') t_{12}}{2\pi i} \Big) = exp \Big(- \frac {ln(\mu'/\mu) t_{12}}{2\pi i} \Big) 
\in \mathcal{A}(\uparrow _2)$ is placed at each $Y$-vertex and minimum.

\end{proof}

\section{Corrections - constructing a knotted graph invariant}\label{corrections}

\subsection{Missing moves}
As in the case of knots, the Kontsevich integral is not invariant under certain
deformations that do not change the isotopy class of the framed graph. In the case
of knots, the only such deformation was ``straightening a hump'', and we fixed this
by multiplying with $Z(\raisebox{-2 mm}{\begin{picture}(0,0)%
\includegraphics{humpcircle.pstex}%
\end{picture}%
\setlength{\unitlength}{3947sp}%
\begingroup\makeatletter\ifx\SetFigFont\undefined%
\gdef\SetFigFont#1#2#3#4#5{%
  \reset@font\fontsize{#1}{#2pt}%
  \fontfamily{#3}\fontseries{#4}\fontshape{#5}%
  \selectfont}%
\fi\endgroup%
\begin{picture}(325,321)(1115,-744)
\end{picture}%
})^{-c/2}$.

The situation now is more complex than in the case of knots. $Z_2$ is invariant under deformations
that do not change the number of critical points or the shape of a vertex.
To transform it into a knotted framed graph invariant $Z_3$, we need to
make corrections to create invariance under the following eight moves:
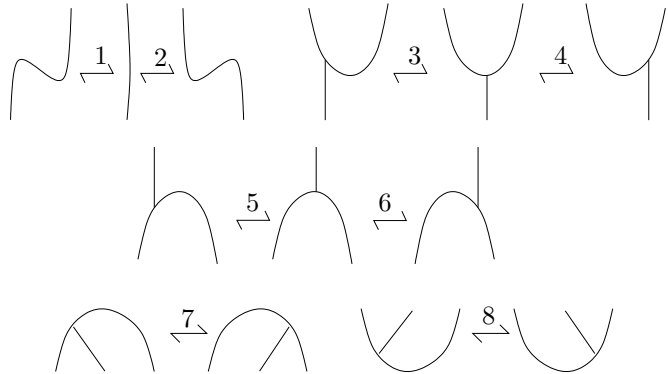
\begin{center}
\input moves.pstex_t

\end{center}

Moves $1$ and $2$ will fix the problem of straightening humps, just like
in the case of knots.

Moves $3$, $4$, $5$ and $6$ guarantee that we can switch a vertex freely
from a $\lambda$ shape to a $Y$ shape and vice versa.

Moves $7$ and $8$ are needed because we excluded vertices that are critical
points at the same time. Any such vertex can be perturbed into a $\lambda$
or a $Y$ shape by an arbitrarily small deformation, but we need the value
of the invariant to be the same whether we push the middle edge over to
the left or to the right.

\smallskip

To implement these corrections, we have four ``correctors'' available: we can
put one each on cups, caps, $Y$-vertices
and $\lambda$-vertices. The ones on cups and caps are elements of $\mathcal{A}(\uparrow _1)$,
the ones on vertices we can think of as elements of $\mathcal{A}(\uparrow _2)$,
as one of the edges can be swept free of chords using the vertex invariance relation.

\smallskip

The eight moves above define eight equations ($1$ and $2$ are equations in $\mathcal{A}_1$,
the rest of them in $\mathcal{A}_2$, the unknowns being the four correctors). The
question is whether this system of equations can be solved.

\subsection{Syzygies}
To reduce the number of equations that the correctors need to satisfy, we
use the syzygies below:

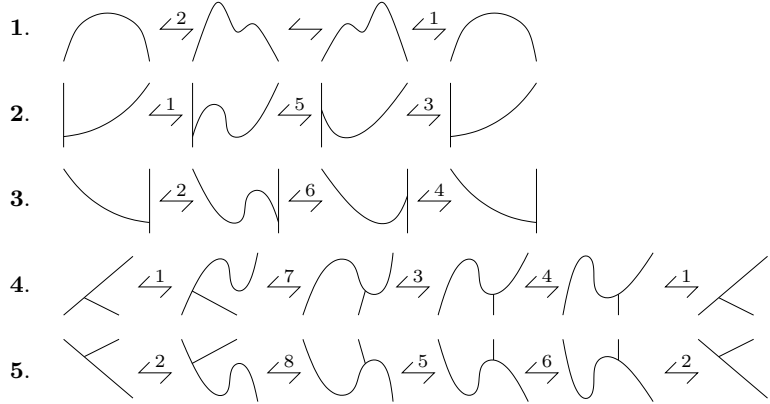
\begin{center}
\input syzygies.pstex_t
\end{center}

The first syzygy implies that once $Z_3$ is invariant under move $1$, it is
automatically invariant under move $2$ (we used this fact already in the knot case).

The second says that invariance under moves $1$ and $3$ implies invariance under
move $5$.

By the third, invariance under moves $2$ and $4$ implies invariance under $6$.

The fourth tells us that fixing move $1$, $3$ and $4$ fixes $7$.

Finally, the fifth syzygy shows that fixing moves $2$, $5$ and $6$ fixes
move $8$ as well.

Therefore, it's enough to make $Z$ invariant under moves $1$, $3$ and $4$, which
together imply invariance under everything else.

\smallskip

\subsection{Translating to equations}
We know already how to make $Z_2$ invariant under move $1$, as this was done in
the knot case.
We need to place a correction $u \in {\mathcal A}_1$ at each minimum, and
a correction $n \in {\mathcal A}_1$ at each maximum, the only equation $u$ and $n$ 
need to satisfy being 
$un = Z_2(\raisebox{-1 mm}{\begin{picture}(0,0)%
\includegraphics{hump.pstex}%
\end{picture}%
\setlength{\unitlength}{3947sp}%
\begingroup\makeatletter\ifx\SetFigFont\undefined%
\gdef\SetFigFont#1#2#3#4#5{%
  \reset@font\fontsize{#1}{#2pt}%
  \fontfamily{#3}\fontseries{#4}\fontshape{#5}%
  \selectfont}%
\fi\endgroup%
\begin{picture}(154,249)(964,-373)
\end{picture}%
})^{-1} \in \mathcal{A}(\uparrow _1)$. 
This will settle move $1$ the same way
it did for knots. (The correction used in the knot case is just an 
instance of the general one described here. The tool used for knots was the lemma
of a distinguished tangle, Lemma \ref{nointeraction}, which said that faraway
parts of a knot don't interact. The proof of this applies in the graph 
context word by word.)

We will denote $Z_2(\raisebox{-1.5 mm}{\begin{picture}(0,0)%
\includegraphics{hump.pstex}%
\end{picture}%
\setlength{\unitlength}{3947sp}%
\begingroup\makeatletter\ifx\SetFigFont\undefined%
\gdef\SetFigFont#1#2#3#4#5{%
  \reset@font\fontsize{#1}{#2pt}%
  \fontfamily{#3}\fontseries{#4}\fontshape{#5}%
  \selectfont}%
\fi\endgroup%
\begin{picture}(154,249)(964,-373)
\end{picture}%
})$ by $\nu$.

\smallskip

Now let us translate moves $3$ and $4$ to equations on
correctors $Y$ and $\lambda$.

For simplicity, let us assume that in the picture for moves
$3$ and $4$, the width of the opening at the top is 1. (This can be 
achieved by horizontal deformations when applying these moves on
any graph.) Let us first also assume that the fixed scale $\mu$ is
also one. (We can correct this later as we know exactly how changing 
$\mu$ effects the value of $Z_2$). 

Finally, we choose one
convenient set of orientations for the edges - all other cases will
follow from this one by performing the orientation change operation
on the appropriate edges.

To compute the values of $Z_2$ on the left side of move $3$ and the
right side of move $4$, we will make use of the following lemma:

\begin{lemma}\label{mirror}
Assuming that the width at the opening at the top of each of the pictures
below is $1$, and $\mu =1$,
$$Z_2 \Big( \raisebox{-0.15 in}{\begin{picture}(0,0)%
\includegraphics{leftieb.pstex}%
\end{picture}%
\setlength{\unitlength}{1302sp}%
\begingroup\makeatletter\ifx\SetFigFont\undefined%
\gdef\SetFigFont#1#2#3#4#5{%
  \reset@font\fontsize{#1}{#2pt}%
  \fontfamily{#3}\fontseries{#4}\fontshape{#5}%
  \selectfont}%
\fi\endgroup%
\begin{picture}(1149,1586)(2314,-3135)
\end{picture}%
} \Big)=
\raisebox{-0.15 in}{\begin{picture}(0,0)%
\includegraphics{leftieb.pstex}%
\end{picture}%
\setlength{\unitlength}{1302sp}%
\begingroup\makeatletter\ifx\SetFigFont\undefined%
\gdef\SetFigFont#1#2#3#4#5{%
  \reset@font\fontsize{#1}{#2pt}%
  \fontfamily{#3}\fontseries{#4}\fontshape{#5}%
  \selectfont}%
\fi\endgroup%
\begin{picture}(1149,1586)(2314,-3135)
\end{picture}%
}, \thickspace
Z_2 \Big( \raisebox{-0.15 in}{\begin{picture}(0,0)%
\includegraphics{rightieb.pstex}%
\end{picture}%
\setlength{\unitlength}{1263sp}%
\begingroup\makeatletter\ifx\SetFigFont\undefined%
\gdef\SetFigFont#1#2#3#4#5{%
  \reset@font\fontsize{#1}{#2pt}%
  \fontfamily{#3}\fontseries{#4}\fontshape{#5}%
  \selectfont}%
\fi\endgroup%
\begin{picture}(1150,1614)(964,-3163)
\end{picture}%
} \Big)= 
\raisebox{-0.15 in}{\begin{picture}(0,0)%
\includegraphics{rightieb.pstex}%
\end{picture}%
\setlength{\unitlength}{1263sp}%
\begingroup\makeatletter\ifx\SetFigFont\undefined%
\gdef\SetFigFont#1#2#3#4#5{%
  \reset@font\fontsize{#1}{#2pt}%
  \fontfamily{#3}\fontseries{#4}\fontshape{#5}%
  \selectfont}%
\fi\endgroup%
\begin{picture}(1150,1614)(964,-3163)
\end{picture}%
} ,$$
i.e. the value of $Z_2$ on these graphs has no chords at all.
\end{lemma}

\begin{proof}
There are two types of chords appearing in 
$Z_2 \Big( \raisebox{-0.15 in}{\begin{picture}(0,0)%
\includegraphics{leftieb.pstex}%
\end{picture}%
\setlength{\unitlength}{1302sp}%
\begingroup\makeatletter\ifx\SetFigFont\undefined%
\gdef\SetFigFont#1#2#3#4#5{%
  \reset@font\fontsize{#1}{#2pt}%
  \fontfamily{#3}\fontseries{#4}\fontshape{#5}%
  \selectfont}%
\fi\endgroup%
\begin{picture}(1149,1586)(2314,-3135)
\end{picture}%
} \Big)$:
\begin{itemize}
\item Horizontal chords connecting the left vertical strand to the diagonal edge
\item Horizontal chords connecting the diagonal edge to the right vertical strand
\end{itemize}

Note that there are no chords on the bottom part: the opening of the strands is
of width $1$, and so is the renormalization scale for the minimum, 
so when computing $Z_2$, we get
$exp \Big( \frac{ln(1) t_{12}}{2 \pi i}\Big)$ as we showed in 
Lemma \ref{rescale2}, which equals $1 \in {\mathcal A}_2$ 
(since $ln(1)=0$), meaning no chords at all.

Also, there are no chords connecting the left and right vertical strands, as these are
parallel, so in the definition of $Z$, $d(z_i-z_i')=0$.

\smallskip

The key observation we will use is that the two types of chords commute.
This is an elementary computation making repeated use of the Locality Lemma \ref{loc}
and the VI relation. What we will prove is that any chord diagram
with chords as above is equivalent to one where all the horizontal chords on the right are 
on the top, followed by all the horizontal chords on the left at the bottom:
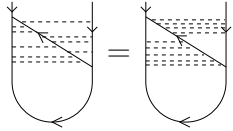
\begin{center}
$\begin{picture}(0,0)%
\includegraphics{pic1b.pstex}%
\end{picture}%
\setlength{\unitlength}{1658sp}%
\begingroup\makeatletter\ifx\SetFigFont\undefined%
\gdef\SetFigFont#1#2#3#4#5{%
  \reset@font\fontsize{#1}{#2pt}%
  \fontfamily{#3}\fontseries{#4}\fontshape{#5}%
  \selectfont}%
\fi\endgroup%
\begin{picture}(1374,1899)(2314,-3448)
\end{picture}%
 \thickspace \raisebox{0.35 in}{=} 
\thickspace \begin{picture}(0,0)%
\includegraphics{pic2b.pstex}%
\end{picture}%
\setlength{\unitlength}{1658sp}%
\begingroup\makeatletter\ifx\SetFigFont\undefined%
\gdef\SetFigFont#1#2#3#4#5{%
  \reset@font\fontsize{#1}{#2pt}%
  \fontfamily{#3}\fontseries{#4}\fontshape{#5}%
  \selectfont}%
\fi\endgroup%
\begin{picture}(1374,1899)(2314,-3448)
\end{picture}%
 $
\end{center}  
To prove this, let us first establish a few basic equalities.

For the bottom chord on the right, we have the following:

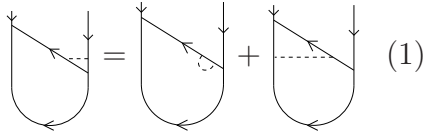
\begin{center}
$\begin{picture}(0,0)%
\includegraphics{pic3b.pstex}%
\end{picture}%
\setlength{\unitlength}{1579sp}%
\begingroup\makeatletter\ifx\SetFigFont\undefined%
\gdef\SetFigFont#1#2#3#4#5{%
  \reset@font\fontsize{#1}{#2pt}%
  \fontfamily{#3}\fontseries{#4}\fontshape{#5}%
  \selectfont}%
\fi\endgroup%
\begin{picture}(1374,1899)(2314,-3448)
\end{picture}%
 \thickspace \raisebox{0.35 in}{=}
\thickspace \begin{picture}(0,0)%
\includegraphics{pic4b.pstex}%
\end{picture}%
\setlength{\unitlength}{1697sp}%
\begingroup\makeatletter\ifx\SetFigFont\undefined%
\gdef\SetFigFont#1#2#3#4#5{%
  \reset@font\fontsize{#1}{#2pt}%
  \fontfamily{#3}\fontseries{#4}\fontshape{#5}%
  \selectfont}%
\fi\endgroup%
\begin{picture}(1374,1899)(2314,-3448)
\end{picture}%

\medspace \raisebox{0.35 in}{+} \medspace \begin{picture}(0,0)%
\includegraphics{pic5b.pstex}%
\end{picture}%
\setlength{\unitlength}{1658sp}%
\begingroup\makeatletter\ifx\SetFigFont\undefined%
\gdef\SetFigFont#1#2#3#4#5{%
  \reset@font\fontsize{#1}{#2pt}%
  \fontfamily{#3}\fontseries{#4}\fontshape{#5}%
  \selectfont}%
\fi\endgroup%
\begin{picture}(1374,1899)(2314,-3448)
\end{picture}%
  \thickspace 
\thickspace \thickspace
\raisebox{0.35 in}{(1)}$
\end{center}
This is true since we can slide the right end of the chord down, then
use the VI relation.

We will use the following shorthand notation:

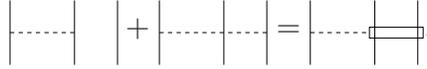
\begin{center}
$\begin{picture}(0,0)%
\includegraphics{short1.pstex}%
\end{picture}%
\setlength{\unitlength}{1776sp}%
\begingroup\makeatletter\ifx\SetFigFont\undefined%
\gdef\SetFigFont#1#2#3#4#5{%
  \reset@font\fontsize{#1}{#2pt}%
  \fontfamily{#3}\fontseries{#4}\fontshape{#5}%
  \selectfont}%
\fi\endgroup%
\begin{picture}(1524,924)(889,-973)
\end{picture}%
  \medspace
\raisebox{0.15 in}{+} \thickspace  \begin{picture}(0,0)%
\includegraphics{short2.pstex}%
\end{picture}%
\setlength{\unitlength}{1776sp}%
\begingroup\makeatletter\ifx\SetFigFont\undefined%
\gdef\SetFigFont#1#2#3#4#5{%
  \reset@font\fontsize{#1}{#2pt}%
  \fontfamily{#3}\fontseries{#4}\fontshape{#5}%
  \selectfont}%
\fi\endgroup%
\begin{picture}(1524,924)(889,-973)
\end{picture}%
 
\thickspace \raisebox{0.15 in}{=} \thickspace  \begin{picture}(0,0)%
\includegraphics{short3.pstex}%
\end{picture}%
\setlength{\unitlength}{1776sp}%
\begingroup\makeatletter\ifx\SetFigFont\undefined%
\gdef\SetFigFont#1#2#3#4#5{%
  \reset@font\fontsize{#1}{#2pt}%
  \fontfamily{#3}\fontseries{#4}\fontshape{#5}%
  \selectfont}%
\fi\endgroup%
\begin{picture}(1599,924)(889,-973)
\end{picture}%
  
\raisebox{0.15 in}{,} $
\end{center}
so the sum above will be denoted as \raisebox{-0.3in}{\begin{picture}(0,0)%
\includegraphics{short4b.pstex}%
\end{picture}%
\setlength{\unitlength}{1934sp}%
\begingroup\makeatletter\ifx\SetFigFont\undefined%
\gdef\SetFigFont#1#2#3#4#5{%
  \reset@font\fontsize{#1}{#2pt}%
  \fontfamily{#3}\fontseries{#4}\fontshape{#5}%
  \selectfont}%
\fi\endgroup%
\begin{picture}(1168,1614)(2314,-3163)
\end{picture}%
} .
We will refer to the fixed end of the chords as the {\it root}.

Next, since a chord on the left commutes with both chords in this sum
(obviously with the horizontal one), and by the Locality Lemma \ref{loc} with
the short one, it commutes with the whole sum:
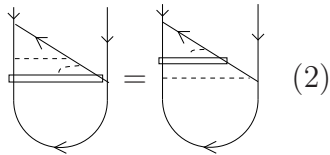
\begin{center}
$\begin{picture}(0,0)%
\includegraphics{pic6b.pstex}%
\end{picture}%
\setlength{\unitlength}{1934sp}%
\begingroup\makeatletter\ifx\SetFigFont\undefined%
\gdef\SetFigFont#1#2#3#4#5{%
  \reset@font\fontsize{#1}{#2pt}%
  \fontfamily{#3}\fontseries{#4}\fontshape{#5}%
  \selectfont}%
\fi\endgroup%
\begin{picture}(1374,1899)(2314,-3448)
\end{picture}%
 \thickspace \raisebox{0.35 in}{=}
\thickspace \begin{picture}(0,0)%
\includegraphics{pic7b.pstex}%
\end{picture}%
\setlength{\unitlength}{1973sp}%
\begingroup\makeatletter\ifx\SetFigFont\undefined%
\gdef\SetFigFont#1#2#3#4#5{%
  \reset@font\fontsize{#1}{#2pt}%
  \fontfamily{#3}\fontseries{#4}\fontshape{#5}%
  \selectfont}%
\fi\endgroup%
\begin{picture}(1374,1899)(2314,-3448)
\end{picture}%
 \thickspace \thickspace
\thickspace \raisebox{0.35 in}{(2)}$
\end{center}

We can also pull only the box part of a sum over a left chord:
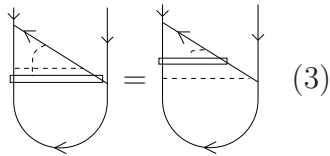
\begin{center}
$\begin{picture}(0,0)%
\includegraphics{pic65b.pstex}%
\end{picture}%
\setlength{\unitlength}{1934sp}%
\begingroup\makeatletter\ifx\SetFigFont\undefined%
\gdef\SetFigFont#1#2#3#4#5{%
  \reset@font\fontsize{#1}{#2pt}%
  \fontfamily{#3}\fontseries{#4}\fontshape{#5}%
  \selectfont}%
\fi\endgroup%
\begin{picture}(1374,1899)(2314,-3448)
\end{picture}%
 \thickspace \raisebox{0.35in}{=}
\thickspace \begin{picture}(0,0)%
\includegraphics{pic7b.pstex}%
\end{picture}%
\setlength{\unitlength}{1973sp}%
\begingroup\makeatletter\ifx\SetFigFont\undefined%
\gdef\SetFigFont#1#2#3#4#5{%
  \reset@font\fontsize{#1}{#2pt}%
  \fontfamily{#3}\fontseries{#4}\fontshape{#5}%
  \selectfont}%
\fi\endgroup%
\begin{picture}(1374,1899)(2314,-3448)
\end{picture}%
 \thickspace \thickspace
\thickspace \raisebox{0.35 in}{(3)} $
\end{center}
This, when expanded, is just a 4T relation (two terms on each side).

And lastly we claim that we can separate nested sums:
\begin{center}
$\begin{picture}(0,0)%
\includegraphics{pic8b.pstex}%
\end{picture}%
\setlength{\unitlength}{1934sp}%
\begingroup\makeatletter\ifx\SetFigFont\undefined%
\gdef\SetFigFont#1#2#3#4#5{%
  \reset@font\fontsize{#1}{#2pt}%
  \fontfamily{#3}\fontseries{#4}\fontshape{#5}%
  \selectfont}%
\fi\endgroup%
\begin{picture}(1374,1899)(2314,-3448)
\end{picture}%
 \thickspace \raisebox{0.35 in}{=} \thickspace 
\begin{picture}(0,0)%
\includegraphics{pic9b.pstex}%
\end{picture}%
\setlength{\unitlength}{1934sp}%
\begingroup\makeatletter\ifx\SetFigFont\undefined%
\gdef\SetFigFont#1#2#3#4#5{%
  \reset@font\fontsize{#1}{#2pt}%
  \fontfamily{#3}\fontseries{#4}\fontshape{#5}%
  \selectfont}%
\fi\endgroup%
\begin{picture}(1374,1899)(2314,-3448)
\end{picture}%
 \thickspace \thickspace
\thickspace \raisebox{0.35 in}{(4)} \raisebox{0.35 in}{,} $
\end{center}
as we can slide the lower box over the sum in the middle, since it commutes with both parts
by the equation above and the Locality Lemma \ref{loc}.

Therefore, we can pull all the chords on the right to the top like we wanted to, by the following
algorithm:

First, we pull all the right ends of the chords over the right vertex, 
using $(1)$, creating a number of boxes. Then we pull the box end
of what was the lowest chord (which has the highest box now) up to its root, over left horizontal 
chords if needed, which we can do using $(3)$.
Continue by pulling the next box up to its root, over horizontal
left chords and the entire lowest sum, both of which are legal steps from above 
($(3)$ and $(4)$).
Continue doing this until all the boxes are united with their roots. 
Then we can slide all the sums to the top as they commute with the 
horizontal left chords, as shown
in $(2)$. Now we reproduce the above procedure backwards, to turn the 
boxes back into right chords, except that, now, they are sitting 
above the left chords. This concludes the proof of the
observation.

Now it's easy to show that all these chords indeed cancel out in 
$Z_2 \Big( \raisebox{-0.15 in}{\begin{picture}(0,0)%
\includegraphics{leftieb.pstex}%
\end{picture}%
\setlength{\unitlength}{1302sp}%
\begingroup\makeatletter\ifx\SetFigFont\undefined%
\gdef\SetFigFont#1#2#3#4#5{%
  \reset@font\fontsize{#1}{#2pt}%
  \fontfamily{#3}\fontseries{#4}\fontshape{#5}%
  \selectfont}%
\fi\endgroup%
\begin{picture}(1149,1586)(2314,-3135)
\end{picture}%
} \Big)$ : 
since the left ones all commute with the right ones, the value 
doesn't change if we compute $Z_2$ separately on the left and on
the right. As the opening of strands is 1, as well as the renormalization scale,
we have $exp \Big( \frac{ln(1) t_{12}}{2 \Pi i} \Big) =1$ on both sides, 
by Lemma \ref{rescale2}.
This concludes the proof. (The proof for the mirror image is the same.)

\end{proof}

An easy corollary is the following:

\begin{lemma}\label{a}
Again assuming that the width at the opening at the top of each of the pictures
below is $1$, and $\mu =1$, the following is true:
$$Z_2 \Big( \raisebox{-0.25 in}{\begin{picture}(0,0)%
\includegraphics{leftie.pstex}%
\end{picture}%
\setlength{\unitlength}{1263sp}%
\begingroup\makeatletter\ifx\SetFigFont\undefined%
\gdef\SetFigFont#1#2#3#4#5{%
  \reset@font\fontsize{#1}{#2pt}%
  \fontfamily{#3}\fontseries{#4}\fontshape{#5}%
  \selectfont}%
\fi\endgroup%
\begin{picture}(1374,2109)(2314,-3658)
\end{picture}%
} \Big)=
\thickspace \raisebox{-0.25 in}{\begin{picture}(0,0)%
\includegraphics{answer.pstex}%
\end{picture}%
\setlength{\unitlength}{1302sp}%
\begingroup\makeatletter\ifx\SetFigFont\undefined%
\gdef\SetFigFont#1#2#3#4#5{%
  \reset@font\fontsize{#1}{#2pt}%
  \fontfamily{#3}\fontseries{#4}\fontshape{#5}%
  \selectfont}%
\fi\endgroup%
\begin{picture}(1224,2052)(7189,-3601)
\put(7726,-3511){\makebox(0,0)[lb]{\smash{{\SetFigFont{8}{9.6}{\rmdefault}{\mddefault}{\updefault}{\color[rgb]{0,0,0}$a$}%
}}}}
\end{picture}%
} \thickspace=
Z_2 \Big( \raisebox{-0.25 in}{\begin{picture}(0,0)%
\includegraphics{rightie.pstex}%
\end{picture}%
\setlength{\unitlength}{1263sp}%
\begingroup\makeatletter\ifx\SetFigFont\undefined%
\gdef\SetFigFont#1#2#3#4#5{%
  \reset@font\fontsize{#1}{#2pt}%
  \fontfamily{#3}\fontseries{#4}\fontshape{#5}%
  \selectfont}%
\fi\endgroup%
\begin{picture}(1374,2109)(964,-3658)
\end{picture}%
} \Big) ,$$

where $a=Z_2(\raisebox{-0.1 in}{\begin{picture}(0,0)%
\includegraphics{adef.pstex}%
\end{picture}%
\setlength{\unitlength}{1421sp}%
\begingroup\makeatletter\ifx\SetFigFont\undefined%
\gdef\SetFigFont#1#2#3#4#5{%
  \reset@font\fontsize{#1}{#2pt}%
  \fontfamily{#3}\fontseries{#4}\fontshape{#5}%
  \selectfont}%
\fi\endgroup%
\begin{picture}(1224,699)(1189,-1348)
\end{picture}%
}) \in {\mathcal A}(\uparrow_2).$
\end{lemma}

Let us say a word about our slight abuse of notation in stating this lemma:

In the statement, we mean that the edges of the answer are oriented
according to the orientations of the edges of the picture we're computing
$Z_2$ of, so the leftmost and rightmost sides are not really equal.

For the definition of $a \in {\mathcal A}_2$, due to the symmetry of
the picture, it doesn't matter which way the two parallel edges go, as
long as they are oriented the same way. In other words, $S_1 S_2 (a)=a$,
where $S_1$ and $S_2$ are the orientation reversing maps. Note also, that by
definition of $Z_2$, $a$ is an invertible element of ${\mathcal A}_2$.

\begin{proof}
First note that
$$Z_2(\raisebox{-0.12 in}{\begin{picture}(0,0)%
\includegraphics{oldadef.pstex}%
\end{picture}%
\setlength{\unitlength}{1224sp}%
\begingroup\makeatletter\ifx\SetFigFont\undefined%
\gdef\SetFigFont#1#2#3#4#5{%
  \reset@font\fontsize{#1}{#2pt}%
  \fontfamily{#3}\fontseries{#4}\fontshape{#5}%
  \selectfont}%
\fi\endgroup%
\begin{picture}(1224,999)(1189,-1348)
\end{picture}%
}) = 
Z_2( \begin{picture}(0,0)%
\includegraphics{2verts.pstex}%
\end{picture}%
\setlength{\unitlength}{1224sp}%
\begingroup\makeatletter\ifx\SetFigFont\undefined%
\gdef\SetFigFont#1#2#3#4#5{%
  \reset@font\fontsize{#1}{#2pt}%
  \fontfamily{#3}\fontseries{#4}\fontshape{#5}%
  \selectfont}%
\fi\endgroup%
\begin{picture}(1524,624)(1039,-973)
\end{picture}%
 )
Z_2(\raisebox{-0.1 in}{\begin{picture}(0,0)%
\includegraphics{adef.pstex}%
\end{picture}%
\setlength{\unitlength}{1421sp}%
\begingroup\makeatletter\ifx\SetFigFont\undefined%
\gdef\SetFigFont#1#2#3#4#5{%
  \reset@font\fontsize{#1}{#2pt}%
  \fontfamily{#3}\fontseries{#4}\fontshape{#5}%
  \selectfont}%
\fi\endgroup%
\begin{picture}(1224,699)(1189,-1348)
\end{picture}%
}),$$
\noindent
since $\mu=1$ and by Lemma \ref{nointeraction}. The first factor on the right
side is $1$, since $\mu=1$.

Now the proof is essentially identical to the previous one. For the first step
of the key observation, we need to slide the right end of the lowest right 
chord over the bottom part of the picture. This is done by using one more
VI relation. The resulting sum of two chords commutes with $a$ by the Locality
Lemma \ref{loc}. Thus we can pull it over $a$ and use the VI relation once more
to get a single chord again.

A few edge orientations differ from the ones in Lemma \ref{mirror}. (We chose orientations 
there to avoid negative signs, here we're switching to the ones we will use later.) 
We can repeat the proof of Lemma \ref{mirror} with the same edge orientations, then switch
the ones we need to switch at the end. This causes no problems, as $Z_2$ commutes with 
the orientation switch operation. Since no chords end on the edges we're reorienting, 
the result doesn't change sign.

\end{proof}

A similar lemma, which is a crucial ingredient in the Murakami-Ohtsuki construction, 
\cite{murakami}, is proved by Le, Murakami, Murakami and Ohtsuki in
\cite{lmmo}; the proof involves the computation of the 
Z-value of an associator in terms of values of the multiple Zeta function.
Note that the element we call $a$ is called $b$ in \cite{lmmo}, and vice versa.

We can now use Lemma \ref{a} to easily compute the value of $Z_2$ on the
trivalent tangles that appear in moves $3$ and $4$. This is done in the following
two corollaries:

\begin{corollary} \label{compute} Still assuming that $\mu=1$,
$Z_2 \Big( \medspace \raisebox{-0.15 in}{\begin{picture}(0,0)%
\includegraphics{leftear.pstex}%
\end{picture}%
\setlength{\unitlength}{868sp}%
\begingroup\makeatletter\ifx\SetFigFont\undefined%
\gdef\SetFigFont#1#2#3#4#5{%
  \reset@font\fontsize{#1}{#2pt}%
  \fontfamily{#3}\fontseries{#4}\fontshape{#5}%
  \selectfont}%
\fi\endgroup%
\begin{picture}(1719,2814)(1009,-2398)
\end{picture}%
} \medspace \Big)
= \thickspace \raisebox{-0.15 in}{\begin{picture}(0,0)%
\includegraphics{leftans.pstex}%
\end{picture}%
\setlength{\unitlength}{868sp}%
\begingroup\makeatletter\ifx\SetFigFont\undefined%
\gdef\SetFigFont#1#2#3#4#5{%
  \reset@font\fontsize{#1}{#2pt}%
  \fontfamily{#3}\fontseries{#4}\fontshape{#5}%
  \selectfont}%
\fi\endgroup%
\begin{picture}(2078,3234)(2300,-2623)
\put(3331,-1996){\makebox(0,0)[rb]{\smash{{\SetFigFont{10}{12.0}{\rmdefault}{\mddefault}{\updefault}{\color[rgb]{0,0,0}$a$}%
}}}}
\end{picture}%
}$,
$Z_2 \Big( \medspace \raisebox{-0.15 in}{\begin{picture}(0,0)%
\includegraphics{rightear.pstex}%
\end{picture}%
\setlength{\unitlength}{868sp}%
\begingroup\makeatletter\ifx\SetFigFont\undefined%
\gdef\SetFigFont#1#2#3#4#5{%
  \reset@font\fontsize{#1}{#2pt}%
  \fontfamily{#3}\fontseries{#4}\fontshape{#5}%
  \selectfont}%
\fi\endgroup%
\begin{picture}(1719,3369)(2704,-2758)
\end{picture}%
} \medspace \Big)
= \thickspace \raisebox{-0.15 in}{\begin{picture}(0,0)%
\includegraphics{rightans.pstex}%
\end{picture}%
\setlength{\unitlength}{868sp}%
\begingroup\makeatletter\ifx\SetFigFont\undefined%
\gdef\SetFigFont#1#2#3#4#5{%
  \reset@font\fontsize{#1}{#2pt}%
  \fontfamily{#3}\fontseries{#4}\fontshape{#5}%
  \selectfont}%
\fi\endgroup%
\begin{picture}(2077,3234)(4354,-2623)
\put(5401,-1996){\rotatebox{360.0}{\makebox(0,0)[lb]{\smash{{\SetFigFont{10}{12.0}{\rmdefault}{\mddefault}{\updefault}{\color[rgb]{0,0,0}$a$}%
}}}}}
\end{picture}%
}$,
$Z_2 \Big( \medspace \raisebox{-0.15 in}{\begin{picture}(0,0)%
\includegraphics{pic17.pstex}%
\end{picture}%
\setlength{\unitlength}{1263sp}%
\begingroup\makeatletter\ifx\SetFigFont\undefined%
\gdef\SetFigFont#1#2#3#4#5{%
  \reset@font\fontsize{#1}{#2pt}%
  \fontfamily{#3}\fontseries{#4}\fontshape{#5}%
  \selectfont}%
\fi\endgroup%
\begin{picture}(1329,1254)(2314,-3088)
\end{picture}%
} \medspace \Big)
= \thickspace \raisebox{-0.15 in}{\begin{picture}(0,0)%
\includegraphics{pic16.pstex}%
\end{picture}%
\setlength{\unitlength}{1263sp}%
\begingroup\makeatletter\ifx\SetFigFont\undefined%
\gdef\SetFigFont#1#2#3#4#5{%
  \reset@font\fontsize{#1}{#2pt}%
  \fontfamily{#3}\fontseries{#4}\fontshape{#5}%
  \selectfont}%
\fi\endgroup%
\begin{picture}(1359,1239)(2239,-3088)
\put(2341,-2836){\makebox(0,0)[lb]{\smash{{\SetFigFont{12}{14.4}{\rmdefault}{\mddefault}{\updefault}{\color[rgb]{0,0,0}$a$}%
}}}}
\end{picture}%
}$, and
$Z_2 \Big( \medspace \raisebox{-0.15 in}{\begin{picture}(0,0)%
\includegraphics{pic18.pstex}%
\end{picture}%
\setlength{\unitlength}{1263sp}%
\begingroup\makeatletter\ifx\SetFigFont\undefined%
\gdef\SetFigFont#1#2#3#4#5{%
  \reset@font\fontsize{#1}{#2pt}%
  \fontfamily{#3}\fontseries{#4}\fontshape{#5}%
  \selectfont}%
\fi\endgroup%
\begin{picture}(1329,1254)(3709,-3088)
\end{picture}%
} \medspace \Big)
= \thickspace \raisebox{-0.15 in}{\begin{picture}(0,0)%
\includegraphics{pic19.pstex}%
\end{picture}%
\setlength{\unitlength}{1263sp}%
\begingroup\makeatletter\ifx\SetFigFont\undefined%
\gdef\SetFigFont#1#2#3#4#5{%
  \reset@font\fontsize{#1}{#2pt}%
  \fontfamily{#3}\fontseries{#4}\fontshape{#5}%
  \selectfont}%
\fi\endgroup%
\begin{picture}(1359,1239)(3574,-3088)
\put(4831,-2836){\rotatebox{360.0}{\makebox(0,0)[rb]{\smash{{\SetFigFont{12}{14.4}{\rmdefault}{\mddefault}{\updefault}{\color[rgb]{0,0,0}$a$}%
}}}}}
\end{picture}%
}$ .
\end{corollary}

Note that we are abusing notation again: the way we defined $a$, it was an element 
of ${\mathcal A}(\uparrow_2)$ where the strands were horizontal. Due to the symmetry of that
picture though, we can slide $a$ either up the curve on the left, or up the curve
on the right to a vertical position, we will get the same result. This is what we
call $a$ here.

\begin{proof}
Unzipping the vertical edge on the right in \raisebox{-0.25 in}{\begin{picture}(0,0)%
\includegraphics{leftie.pstex}%
\end{picture}%
\setlength{\unitlength}{1263sp}%
\begingroup\makeatletter\ifx\SetFigFont\undefined%
\gdef\SetFigFont#1#2#3#4#5{%
  \reset@font\fontsize{#1}{#2pt}%
  \fontfamily{#3}\fontseries{#4}\fontshape{#5}%
  \selectfont}%
\fi\endgroup%
\begin{picture}(1374,2109)(2314,-3658)
\end{picture}%
},
we get \raisebox{-0.25 in}{\begin{picture}(0,0)%
\includegraphics{pic11.pstex}%
\end{picture}%
\setlength{\unitlength}{1263sp}%
\begingroup\makeatletter\ifx\SetFigFont\undefined%
\gdef\SetFigFont#1#2#3#4#5{%
  \reset@font\fontsize{#1}{#2pt}%
  \fontfamily{#3}\fontseries{#4}\fontshape{#5}%
  \selectfont}%
\fi\endgroup%
\begin{picture}(1374,2109)(2314,-3658)
\end{picture}%
} .

Since $Z_2$ commutes with vertical unzip, and in 
$Z_2 \Big( \raisebox{-0.25 in}{\begin{picture}(0,0)%
\includegraphics{leftie.pstex}%
\end{picture}%
\setlength{\unitlength}{1263sp}%
\begingroup\makeatletter\ifx\SetFigFont\undefined%
\gdef\SetFigFont#1#2#3#4#5{%
  \reset@font\fontsize{#1}{#2pt}%
  \fontfamily{#3}\fontseries{#4}\fontshape{#5}%
  \selectfont}%
\fi\endgroup%
\begin{picture}(1374,2109)(2314,-3658)
\end{picture}%
} \Big)$
no chords end on the edge we're unzipping, we can deduce that
$$Z_2 \Big( \raisebox{-0.25 in}{\begin{picture}(0,0)%
\includegraphics{pic11.pstex}%
\end{picture}%
\setlength{\unitlength}{1263sp}%
\begingroup\makeatletter\ifx\SetFigFont\undefined%
\gdef\SetFigFont#1#2#3#4#5{%
  \reset@font\fontsize{#1}{#2pt}%
  \fontfamily{#3}\fontseries{#4}\fontshape{#5}%
  \selectfont}%
\fi\endgroup%
\begin{picture}(1374,2109)(2314,-3658)
\end{picture}%
} \Big) \thickspace
= \thickspace \raisebox{-0.25 in}{\begin{picture}(0,0)%
\includegraphics{pic12.pstex}%
\end{picture}%
\setlength{\unitlength}{1263sp}%
\begingroup\makeatletter\ifx\SetFigFont\undefined%
\gdef\SetFigFont#1#2#3#4#5{%
  \reset@font\fontsize{#1}{#2pt}%
  \fontfamily{#3}\fontseries{#4}\fontshape{#5}%
  \selectfont}%
\fi\endgroup%
\begin{picture}(1404,2109)(2284,-3658)
\put(2371,-2791){\makebox(0,0)[lb]{\smash{{\SetFigFont{11}{13.2}{\rmdefault}{\mddefault}{\updefault}{\color[rgb]{0,0,0}$a$}%
}}}}
\end{picture}%
} \medspace .$$ 
By the factorization property,
$$Z_2 \Big( \raisebox{-0.25 in}{\begin{picture}(0,0)%
\includegraphics{pic11.pstex}%
\end{picture}%
\setlength{\unitlength}{1263sp}%
\begingroup\makeatletter\ifx\SetFigFont\undefined%
\gdef\SetFigFont#1#2#3#4#5{%
  \reset@font\fontsize{#1}{#2pt}%
  \fontfamily{#3}\fontseries{#4}\fontshape{#5}%
  \selectfont}%
\fi\endgroup%
\begin{picture}(1374,2109)(2314,-3658)
\end{picture}%
} \Big) \thickspace
= \thickspace Z_2\Big( \raisebox{-0.17 in}{\begin{picture}(0,0)%
\includegraphics{pic13.pstex}%
\end{picture}%
\setlength{\unitlength}{1263sp}%
\begingroup\makeatletter\ifx\SetFigFont\undefined%
\gdef\SetFigFont#1#2#3#4#5{%
  \reset@font\fontsize{#1}{#2pt}%
  \fontfamily{#3}\fontseries{#4}\fontshape{#5}%
  \selectfont}%
\fi\endgroup%
\begin{picture}(1374,1539)(2314,-3088)
\end{picture}%
} \Big)
Z_2\Big( \raisebox{-0.10 in}{\begin{picture}(0,0)%
\includegraphics{pic14.pstex}%
\end{picture}%
\setlength{\unitlength}{1263sp}%
\begingroup\makeatletter\ifx\SetFigFont\undefined%
\gdef\SetFigFont#1#2#3#4#5{%
  \reset@font\fontsize{#1}{#2pt}%
  \fontfamily{#3}\fontseries{#4}\fontshape{#5}%
  \selectfont}%
\fi\endgroup%
\begin{picture}(1224,909)(2389,-3658)
\end{picture}%
} \Big).$$
However, the second factor is $1$ (i.e. has no chords) by Lemma \ref{rescale2}, 
since the opening of the strands at the top is $1$, and so is the width of the
renormalization at the bottom.

Therefore, we deduce that $Z_2\Big( \raisebox{-0.17 in}{\begin{picture}(0,0)%
\includegraphics{pic13.pstex}%
\end{picture}%
\setlength{\unitlength}{1263sp}%
\begingroup\makeatletter\ifx\SetFigFont\undefined%
\gdef\SetFigFont#1#2#3#4#5{%
  \reset@font\fontsize{#1}{#2pt}%
  \fontfamily{#3}\fontseries{#4}\fontshape{#5}%
  \selectfont}%
\fi\endgroup%
\begin{picture}(1374,1539)(2314,-3088)
\end{picture}%
} \Big)
= \thickspace \raisebox{-0.17 in}{\begin{picture}(0,0)%
\includegraphics{pic15.pstex}%
\end{picture}%
\setlength{\unitlength}{1263sp}%
\begingroup\makeatletter\ifx\SetFigFont\undefined%
\gdef\SetFigFont#1#2#3#4#5{%
  \reset@font\fontsize{#1}{#2pt}%
  \fontfamily{#3}\fontseries{#4}\fontshape{#5}%
  \selectfont}%
\fi\endgroup%
\begin{picture}(1404,1674)(2284,-3223)
\put(2371,-2791){\makebox(0,0)[lb]{\smash{{\SetFigFont{11}{13.2}{\rmdefault}{\mddefault}{\updefault}{\color[rgb]{0,0,0}$a$}%
}}}}
\end{picture}%
} \medspace .$

Also, we can forget about the vertical strand on the right, as we know from
Lemma \ref{nointeraction} (which generalizes to graphs word by word) that
faraway strands don't interact. This proves the first equality.
The proof or the second equality is the same, with all the pictures mirrored. 

\smallskip

By the factorization property,
$$Z_2 \Big( \medspace \raisebox{-0.25 in}{\begin{picture}(0,0)%
\includegraphics{leftear.pstex}%
\end{picture}%
\setlength{\unitlength}{868sp}%
\begingroup\makeatletter\ifx\SetFigFont\undefined%
\gdef\SetFigFont#1#2#3#4#5{%
  \reset@font\fontsize{#1}{#2pt}%
  \fontfamily{#3}\fontseries{#4}\fontshape{#5}%
  \selectfont}%
\fi\endgroup%
\begin{picture}(1719,2814)(1009,-2398)
\end{picture}%
} \medspace \Big)=
Z_2 \Big( \medspace \raisebox{-0.15 in}{\begin{picture}(0,0)%
\includegraphics{pic17.pstex}%
\end{picture}%
\setlength{\unitlength}{1263sp}%
\begingroup\makeatletter\ifx\SetFigFont\undefined%
\gdef\SetFigFont#1#2#3#4#5{%
  \reset@font\fontsize{#1}{#2pt}%
  \fontfamily{#3}\fontseries{#4}\fontshape{#5}%
  \selectfont}%
\fi\endgroup%
\begin{picture}(1329,1254)(2314,-3088)
\end{picture}%
} \medspace \Big)
Z_2 \Big( \medspace \raisebox{-0.15 in}{\begin{picture}(0,0)%
\includegraphics{lvertex.pstex}%
\end{picture}%
\setlength{\unitlength}{908sp}%
\begingroup\makeatletter\ifx\SetFigFont\undefined%
\gdef\SetFigFont#1#2#3#4#5{%
  \reset@font\fontsize{#1}{#2pt}%
  \fontfamily{#3}\fontseries{#4}\fontshape{#5}%
  \selectfont}%
\fi\endgroup%
\begin{picture}(1359,1659)(2254,-1573)
\end{picture}%
} \medspace \Big).$$
Again, the second factor is $1$ (no chords at all), by Lemma \ref{rescale2}, if we assume
that the width of the opening where we cut is $1$ (which we are free to do of course, by 
horizontal deformations), and the width we use for the vertex renormalization is also $1$.
The third equality then follows, as does the fourth by mirroring the
pictures.

\end{proof}

We are now ready to form the equations corresponding to moves $3$ and $4$ (still
assuming that $\mu=1$). 

The value of the corrected integral $Z_3$ is constructed from the value of $Z_2$
by placing correctors on each vertex and extremum. Let us call these correctors
$n$, $u$, $\lambda$, and $Y$, for maxima, minima, $\lambda$-shaped vertices, and
$Y$-shaped vertices respectively. The correctors $n$ and $u$ are elements of
${\mathcal A}(\uparrow)$, $\lambda$ and $Y$ are elements of 
${\mathcal A}(\uparrow_2)$, and they are placed on the graph as shown in the figure 
(edge orientation issues will be dealt with later):
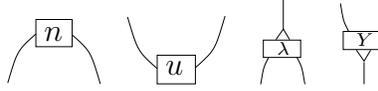
\begin{center}
\begin{picture}(0,0)%
\includegraphics{pic25b.pstex}%
\end{picture}%
\setlength{\unitlength}{1500sp}%
\begingroup\makeatletter\ifx\SetFigFont\undefined%
\gdef\SetFigFont#1#2#3#4#5{%
  \reset@font\fontsize{#1}{#2pt}%
  \fontfamily{#3}\fontseries{#4}\fontshape{#5}%
  \selectfont}%
\fi\endgroup%
\begin{picture}(1554,1089)(1174,-2788)
\put(1786,-2071){\makebox(0,0)[lb]{\smash{{\SetFigFont{12}{14.4}{\rmdefault}{\mddefault}{\updefault}{\color[rgb]{0,0,0}$n$}%
}}}}
\end{picture}%
  \thickspace \thickspace \thickspace
\begin{picture}(0,0)%
\includegraphics{pic23b.pstex}%
\end{picture}%
\setlength{\unitlength}{1579sp}%
\begingroup\makeatletter\ifx\SetFigFont\undefined%
\gdef\SetFigFont#1#2#3#4#5{%
  \reset@font\fontsize{#1}{#2pt}%
  \fontfamily{#3}\fontseries{#4}\fontshape{#5}%
  \selectfont}%
\fi\endgroup%
\begin{picture}(1554,1089)(1174,-1723)
\put(1786,-1576){\makebox(0,0)[lb]{\smash{{\SetFigFont{12}{14.4}{\rmdefault}{\mddefault}{\updefault}{\color[rgb]{0,0,0}$u$}%
}}}}
\end{picture}%
  \thickspace \thickspace \thickspace \thickspace
\begin{picture}(0,0)%
\includegraphics{pic24b.pstex}%
\end{picture}%
\setlength{\unitlength}{868sp}%
\begingroup\makeatletter\ifx\SetFigFont\undefined%
\gdef\SetFigFont#1#2#3#4#5{%
  \reset@font\fontsize{#1}{#2pt}%
  \fontfamily{#3}\fontseries{#4}\fontshape{#5}%
  \selectfont}%
\fi\endgroup%
\begin{picture}(1284,2454)(1249,-3358)
\put(1726,-2506){\makebox(0,0)[lb]{\smash{{\SetFigFont{7}{8.4}{\rmdefault}{\mddefault}{\updefault}{\color[rgb]{0,0,0}$\lambda$}%
}}}}
\end{picture}%
  \thickspace \thickspace \thickspace \thickspace
\begin{picture}(0,0)%
\includegraphics{pic22b.pstex}%
\end{picture}%
\setlength{\unitlength}{908sp}%
\begingroup\makeatletter\ifx\SetFigFont\undefined%
\gdef\SetFigFont#1#2#3#4#5{%
  \reset@font\fontsize{#1}{#2pt}%
  \fontfamily{#3}\fontseries{#4}\fontshape{#5}%
  \selectfont}%
\fi\endgroup%
\begin{picture}(1284,2229)(1249,-2413)
\put(1696,-1366){\makebox(0,0)[lb]{\smash{{\SetFigFont{7}{8.4}{\rmdefault}{\mddefault}{\updefault}{\color[rgb]{0,0,0}$Y$}%
}}}}
\end{picture}%

\end{center}

Using this terminology, we have proven the following proposition:

\begin{proposition} \label{translation}
With corrections $\lambda \in {\mathcal A}_2$, $Y \in {\mathcal A}_2$,
and $u \in {\mathcal A}_1$, such that
$$\raisebox{-0.2 in}{\input pic20.pstex_t} \thickspace = \thickspace
\raisebox{-0.2 in}{\begin{picture}(0,0)%
\includegraphics{yvertex.pstex}%
\end{picture}%
\setlength{\unitlength}{1421sp}%
\begingroup\makeatletter\ifx\SetFigFont\undefined%
\gdef\SetFigFont#1#2#3#4#5{%
  \reset@font\fontsize{#1}{#2pt}%
  \fontfamily{#3}\fontseries{#4}\fontshape{#5}%
  \selectfont}%
\fi\endgroup%
\begin{picture}(1464,2289)(1159,-2473)
\put(1711,-1381){\makebox(0,0)[lb]{\smash{{\SetFigFont{12}{14.4}{\rmdefault}{\mddefault}{\updefault}{\color[rgb]{0,0,0}$Y$}%
}}}}
\end{picture}%
} \thickspace = \thickspace
\raisebox{-0.2 in}{\input pic21.pstex_t} \medspace ,$$
and $n=u^{-1}\nu^{-1}$, $Z_3$
is invariant under moves $1$-$8$, and therefore it is an isotopy
invariant of knotted trivalent graphs. As before, 
$\nu = Z_2(\raisebox{-1.5 mm}{\begin{picture}(0,0)%
\includegraphics{hump.pstex}%
\end{picture}%
\setlength{\unitlength}{3947sp}%
\begingroup\makeatletter\ifx\SetFigFont\undefined%
\gdef\SetFigFont#1#2#3#4#5{%
  \reset@font\fontsize{#1}{#2pt}%
  \fontfamily{#3}\fontseries{#4}\fontshape{#5}%
  \selectfont}%
\fi\endgroup%
\begin{picture}(154,249)(964,-373)
\end{picture}%
})$.
\end{proposition}

Note that we used above that $Z_2$ of a $Y$ vertex with an opening of width $1$
at the top equals $1$, since $\mu=1$.

\subsection{The resulting invariants}

There is an obvious set of corrections satisfying the above equations:
$\lambda = a^{-1}$, $Y = 1$ and $u=1$. This forces 
$n=(Z_2(\raisebox{-1.5 mm}{\begin{picture}(0,0)%
\includegraphics{hump.pstex}%
\end{picture}%
\setlength{\unitlength}{3947sp}%
\begingroup\makeatletter\ifx\SetFigFont\undefined%
\gdef\SetFigFont#1#2#3#4#5{%
  \reset@font\fontsize{#1}{#2pt}%
  \fontfamily{#3}\fontseries{#4}\fontshape{#5}%
  \selectfont}%
\fi\endgroup%
\begin{picture}(154,249)(964,-373)
\end{picture}%
}))^{-1}$.

Note that this $\lambda$ and $Y$ value works only with the orientations
of the edges chosen as in the pictures above. Correct notation would
be to say that $\lambda_{\downarrow \uparrow \uparrow}= a^{-1}$, where
in the subscript the first arrow shows the orientation of the top strand
of the vertex, the second stands for the lower left, and the last stands 
for the lower right strand. 

It's easy to now come up with a complete
set of corrections for all orientations: if we change the orientation
of one of the lower strands, we apply the corresponding
orientation switch operation to the correction:
$\lambda_{\downarrow \downarrow \uparrow}=S_1( a^{-1})$,
$\lambda_{\downarrow \uparrow \downarrow}=S_2( a^{-1})$,
$\lambda_{\downarrow \downarrow \downarrow}=S_1 S_2 (a^{-1})=a^{-1}$.
These satisfy the equations of Proposition \ref{translation}, 
since $Z_2$ commutes with the orientation
switch operation.

If we switch the orientation of all the edges, the proof remains unchanged
(due to the fact that $S_1 S_2(a)=a$), so we have 
$\lambda_{\uparrow \downarrow \downarrow}=a^{-1}$,
and then by the above reasoning, the rest follows:
$\lambda_{\uparrow \uparrow \downarrow}=S_1(a^{-1})$,
$\lambda_{\uparrow \downarrow \uparrow}=S_2(a^{-1})$, and
$\lambda_{\uparrow \uparrow \uparrow}= S_1 S_2 (a^{-1})= a^{-1}$.

It is worth noting that $S_1(a)=S_1 (S_1 S_2 (a))=S_2(a)$.

Since $Y$ is trivial, orientation switches don't affect $Y$.

For $n$ and $u$ the orientation of the strand doesn't matter, since
these are elements in ${\mathcal A}_1$, so each chord has two endings on
the one strand, thus orientation change operation will not change
any signs.

\begin{proposition}
With the complete set of corrections described above, $Z_3$
commutes with the orientation switch, edge delete
and connected sum operation.
\end{proposition}

\begin{proof}
It is obvious that $Z_3$ commutes with the orientation
switch, as $Z_2$ has the same property and 
we designed the corrections ($\lambda$ especially) to
keep it true.  

We only need to deal with vertical edge delete: 
since we now have an isotopy invariant, we can first deform the edge to be deleted
into a straight vertical line with a $Y$ vertex on top and $\lambda$ vertex on bottom.
Thus, it's enough to check that the following equalities are true:

\begin{center}
$\raisebox{-0.15 in}{\begin{picture}(0,0)%
\includegraphics{pic22.pstex}%
\end{picture}%
\setlength{\unitlength}{1381sp}%
\begingroup\makeatletter\ifx\SetFigFont\undefined%
\gdef\SetFigFont#1#2#3#4#5{%
  \reset@font\fontsize{#1}{#2pt}%
  \fontfamily{#3}\fontseries{#4}\fontshape{#5}%
  \selectfont}%
\fi\endgroup%
\begin{picture}(1389,1599)(1174,-1783)
\put(1696,-1366){\makebox(0,0)[lb]{\smash{{\SetFigFont{11}{13.2}{\rmdefault}{\mddefault}{\updefault}{\color[rgb]{0,0,0}$Y$}%
}}}}
\end{picture}%
} \thickspace
= \thickspace \raisebox{-0.1 in}{\begin{picture}(0,0)%
\includegraphics{pic23.pstex}%
\end{picture}%
\setlength{\unitlength}{1579sp}%
\begingroup\makeatletter\ifx\SetFigFont\undefined%
\gdef\SetFigFont#1#2#3#4#5{%
  \reset@font\fontsize{#1}{#2pt}%
  \fontfamily{#3}\fontseries{#4}\fontshape{#5}%
  \selectfont}%
\fi\endgroup%
\begin{picture}(1614,1089)(1114,-1723)
\put(1786,-1576){\makebox(0,0)[lb]{\smash{{\SetFigFont{12}{14.4}{\rmdefault}{\mddefault}{\updefault}{\color[rgb]{0,0,0}$u$}%
}}}}
\end{picture}%
} \medspace , \thickspace \thickspace
\raisebox{-0.15 in}{\begin{picture}(0,0)%
\includegraphics{pic24.pstex}%
\end{picture}%
\setlength{\unitlength}{1381sp}%
\begingroup\makeatletter\ifx\SetFigFont\undefined%
\gdef\SetFigFont#1#2#3#4#5{%
  \reset@font\fontsize{#1}{#2pt}%
  \fontfamily{#3}\fontseries{#4}\fontshape{#5}%
  \selectfont}%
\fi\endgroup%
\begin{picture}(1419,1599)(1174,-3358)
\put(1726,-2506){\makebox(0,0)[lb]{\smash{{\SetFigFont{11}{13.2}{\rmdefault}{\mddefault}{\updefault}{\color[rgb]{0,0,0}$\lambda$}%
}}}}
\end{picture}%
} \thickspace
= \thickspace \raisebox{-0.1 in}{\begin{picture}(0,0)%
\includegraphics{pic25.pstex}%
\end{picture}%
\setlength{\unitlength}{1500sp}%
\begingroup\makeatletter\ifx\SetFigFont\undefined%
\gdef\SetFigFont#1#2#3#4#5{%
  \reset@font\fontsize{#1}{#2pt}%
  \fontfamily{#3}\fontseries{#4}\fontshape{#5}%
  \selectfont}%
\fi\endgroup%
\begin{picture}(1614,1089)(1114,-2788)
\put(1786,-2071){\makebox(0,0)[lb]{\smash{{\SetFigFont{12}{14.4}{\rmdefault}{\mddefault}{\updefault}{\color[rgb]{0,0,0}$n$}%
}}}}
\end{picture}%
} \thickspace ,$
\end{center}
and also for the opposite orientations of the strands.

Each time $\lambda$ appears above, we need to check the equations for all appropriate 
choices of $\lambda$ depending on the orientations of the strands, for example
$\lambda_{\downarrow \uparrow \downarrow}$ and $\lambda_{\uparrow \uparrow \downarrow}$
in the second equation.

Indeed, $\lambda_{\downarrow \uparrow \downarrow}=S_2( a^{-1})=S_2(S_1 S_2 (a^{-1})=S_1(a^{-1})=
\lambda_{\uparrow \uparrow \downarrow}$, and
\begin{center}
$ \raisebox{-0.2 in}{\begin{picture}(0,0)%
\includegraphics{pic26.pstex}%
\end{picture}%
\setlength{\unitlength}{1539sp}%
\begingroup\makeatletter\ifx\SetFigFont\undefined%
\gdef\SetFigFont#1#2#3#4#5{%
  \reset@font\fontsize{#1}{#2pt}%
  \fontfamily{#3}\fontseries{#4}\fontshape{#5}%
  \selectfont}%
\fi\endgroup%
\begin{picture}(1636,1599)(1159,-3358)
\put(1426,-2446){\makebox(0,0)[lb]{\smash{{\SetFigFont{9}{10.8}{\rmdefault}{\mddefault}{\updefault}{\color[rgb]{0,0,0}$S_1(a)$}%
}}}}
\end{picture}%
} \thickspace
= \thickspace Z_2( \raisebox{-0.1 in}{\begin{picture}(0,0)%
\includegraphics{pic27b.pstex}%
\end{picture}%
\setlength{\unitlength}{1224sp}%
\begingroup\makeatletter\ifx\SetFigFont\undefined%
\gdef\SetFigFont#1#2#3#4#5{%
  \reset@font\fontsize{#1}{#2pt}%
  \fontfamily{#3}\fontseries{#4}\fontshape{#5}%
  \selectfont}%
\fi\endgroup%
\begin{picture}(1264,774)(1189,-1423)
\end{picture}%
}) \thickspace
= \medspace Z_2( \raisebox{-0.1 in}{\begin{picture}(0,0)%
\includegraphics{pic27c.pstex}%
\end{picture}%
\setlength{\unitlength}{1224sp}%
\begingroup\makeatletter\ifx\SetFigFont\undefined%
\gdef\SetFigFont#1#2#3#4#5{%
  \reset@font\fontsize{#1}{#2pt}%
  \fontfamily{#3}\fontseries{#4}\fontshape{#5}%
  \selectfont}%
\fi\endgroup%
\begin{picture}(1194,765)(1204,-1423)
\end{picture}%
}) \medspace
= \medspace Z_2( \raisebox{-2 mm}{\begin{picture}(0,0)%
\includegraphics{humpcircleb.pstex}%
\end{picture}%
\setlength{\unitlength}{3947sp}%
\begingroup\makeatletter\ifx\SetFigFont\undefined%
\gdef\SetFigFont#1#2#3#4#5{%
  \reset@font\fontsize{#1}{#2pt}%
  \fontfamily{#3}\fontseries{#4}\fontshape{#5}%
  \selectfont}%
\fi\endgroup%
\begin{picture}(325,321)(1418,-425)
\end{picture}%
}) \medspace 
= \medspace Z_2(\raisebox{-1.5 mm}{\begin{picture}(0,0)%
\includegraphics{hump.pstex}%
\end{picture}%
\setlength{\unitlength}{3947sp}%
\begingroup\makeatletter\ifx\SetFigFont\undefined%
\gdef\SetFigFont#1#2#3#4#5{%
  \reset@font\fontsize{#1}{#2pt}%
  \fontfamily{#3}\fontseries{#4}\fontshape{#5}%
  \selectfont}%
\fi\endgroup%
\begin{picture}(154,249)(964,-373)
\end{picture}%
}),$
\end{center}
where all equalities are understood in ${\mathcal A}(\uparrow_1)$, 
using the isomorphism ${\mathcal A}(\uparrow_1) \cong {\mathcal A}(S^1)$. 
The first equality is true by definition of $a$ and the fact that $Z_2$ 
commutes with orientation switches, using that $\mu =1$ and 
Lemma \ref{nointeraction}. The second one is again $\mu = 1$ and 
Lemma \ref{nointeraction}, while the third is horizontal deformations and 
moving critical points, and the fourth is due to Lemma \ref{nointeraction}. 
Now commutativity with edge deletion follows by 
taking inverses on both sides. There is nothing to check for $Y$ and $u$, 
as they're both trivial.

Connected sum is the ``reverse'' of edge delete, the fact that $Z_3$ commutes with it 
is proved by backtracking
the above proof, using that the connected sum is well-defined.
\end{proof}

Before we deal with the unzip operation, let us produce a one parameter family of
corrections that all yield knotted graph invariants well behaved with respect to 
the orientation switch, edge delete, and connected sum operations. The reader who is
satisfied with one invariant is welcome to skip ahead to the paragraph following
Remark \ref{mo}.

The statement of the following lemma appears in \cite{murakami}, the proof
there is phrased in terms of q-tangles, but is based on the same trick.

\begin{lemma}
$\raisebox{-0.2 in}{\input pic40b.pstex_t} \thickspace
= \thickspace \raisebox{-0.2 in}{\begin{picture}(0,0)%
\includegraphics{pic41.pstex}%
\end{picture}%
\setlength{\unitlength}{1145sp}%
\begingroup\makeatletter\ifx\SetFigFont\undefined%
\gdef\SetFigFont#1#2#3#4#5{%
  \reset@font\fontsize{#1}{#2pt}%
  \fontfamily{#3}\fontseries{#4}\fontshape{#5}%
  \selectfont}%
\fi\endgroup%
\begin{picture}(1990,2334)(874,-1888)
\put(1111,-691){\makebox(0,0)[lb]{\smash{{\SetFigFont{9}{10.8}{\rmdefault}{\mddefault}{\updefault}{\color[rgb]{0,0,0}$b^{-x}$}%
}}}}
\end{picture}%
}$, where we
define  $b \in {\mathcal A}_2$ to be $b=Z_2(\raisebox{-0.03 in}{\begin{picture}(0,0)%
\includegraphics{bdef.pstex}%
\end{picture}%
\setlength{\unitlength}{1224sp}%
\begingroup\makeatletter\ifx\SetFigFont\undefined%
\gdef\SetFigFont#1#2#3#4#5{%
  \reset@font\fontsize{#1}{#2pt}%
  \fontfamily{#3}\fontseries{#4}\fontshape{#5}%
  \selectfont}%
\fi\endgroup%
\begin{picture}(1224,699)(1189,-2173)
\end{picture}%
} )$,
an ``upside down $a$'', and $x \in {\mathbb R}$ any real number.
\end{lemma}

\begin{proof}
Throughout the proof we will assume that all strand openings are
of width $1$ as well as $\mu =1$, however, the statement itself is just 
an equality in ${\mathcal A}(Y)$, and thus independent of the
choice of $\mu$.

Then, the same way we had done in Lemma \ref{mirror} to Corollary \ref{compute}, 
we can compute
$$Z_2 \Big ( \raisebox{-0.12 in}{\begin{picture}(0,0)%
\includegraphics{pic29.pstex}%
\end{picture}%
\setlength{\unitlength}{1263sp}%
\begingroup\makeatletter\ifx\SetFigFont\undefined%
\gdef\SetFigFont#1#2#3#4#5{%
  \reset@font\fontsize{#1}{#2pt}%
  \fontfamily{#3}\fontseries{#4}\fontshape{#5}%
  \selectfont}%
\fi\endgroup%
\begin{picture}(1329,1254)(2314,-4438)
\end{picture}%
} \Big )=
\thickspace \raisebox{-0.14 in}{\begin{picture}(0,0)%
\includegraphics{pic30.pstex}%
\end{picture}%
\setlength{\unitlength}{1500sp}%
\begingroup\makeatletter\ifx\SetFigFont\undefined%
\gdef\SetFigFont#1#2#3#4#5{%
  \reset@font\fontsize{#1}{#2pt}%
  \fontfamily{#3}\fontseries{#4}\fontshape{#5}%
  \selectfont}%
\fi\endgroup%
\begin{picture}(1494,1254)(2224,-4438)
\put(2371,-3856){\makebox(0,0)[lb]{\smash{{\SetFigFont{10}{12.0}{\rmdefault}{\mddefault}{\updefault}{\color[rgb]{0,0,0}$b$}%
}}}}
\end{picture}%
} \medspace .$$

By multiplicativity, the fact that faraway strands don't interact 
(Lemma \ref{nointeraction}), and our previous computations,
$$ Z_2 \Bigg( \raisebox{-0.5 in}{\begin{picture}(0,0)%
\includegraphics{pic31.pstex}%
\end{picture}%
\setlength{\unitlength}{1066sp}%
\begingroup\makeatletter\ifx\SetFigFont\undefined%
\gdef\SetFigFont#1#2#3#4#5{%
  \reset@font\fontsize{#1}{#2pt}%
  \fontfamily{#3}\fontseries{#4}\fontshape{#5}%
  \selectfont}%
\fi\endgroup%
\begin{picture}(2769,4974)(1009,-2773)
\end{picture}%
} \Bigg)=
\thickspace \raisebox{-0.5 in}{\input pic33.pstex_t} \medspace . $$

By an unzip, it follows that
$$ Z_2 \Bigg( \raisebox{-0.35 in}{\begin{picture}(0,0)%
\includegraphics{pic32.pstex}%
\end{picture}%
\setlength{\unitlength}{1145sp}%
\begingroup\makeatletter\ifx\SetFigFont\undefined%
\gdef\SetFigFont#1#2#3#4#5{%
  \reset@font\fontsize{#1}{#2pt}%
  \fontfamily{#3}\fontseries{#4}\fontshape{#5}%
  \selectfont}%
\fi\endgroup%
\begin{picture}(1659,3639)(814,-1288)
\end{picture}%
} \Bigg)=
\thickspace \raisebox{-0.35 in}{\input pic34.pstex_t} \medspace . $$

We know that ``adding a hump'' amounts to multiplying by a 
factor of $\nu=Z_2(\raisebox{-0.05 in}{\begin{picture}(0,0)%
\includegraphics{hump.pstex}%
\end{picture}%
\setlength{\unitlength}{3947sp}%
\begingroup\makeatletter\ifx\SetFigFont\undefined%
\gdef\SetFigFont#1#2#3#4#5{%
  \reset@font\fontsize{#1}{#2pt}%
  \fontfamily{#3}\fontseries{#4}\fontshape{#5}%
  \selectfont}%
\fi\endgroup%
\begin{picture}(154,249)(964,-373)
\end{picture}%
})$, therefore
$$ Z_2 \Big ( \raisebox{-0.2 in}{\begin{picture}(0,0)%
\includegraphics{pic35.pstex}%
\end{picture}%
\setlength{\unitlength}{908sp}%
\begingroup\makeatletter\ifx\SetFigFont\undefined%
\gdef\SetFigFont#1#2#3#4#5{%
  \reset@font\fontsize{#1}{#2pt}%
  \fontfamily{#3}\fontseries{#4}\fontshape{#5}%
  \selectfont}%
\fi\endgroup%
\begin{picture}(1224,2610)(799,-259)
\end{picture}%
} \Big )=
Z_2 \Big ( \raisebox{-0.15 in}{\begin{picture}(0,0)%
\includegraphics{pic36.pstex}%
\end{picture}%
\setlength{\unitlength}{1105sp}%
\begingroup\makeatletter\ifx\SetFigFont\undefined%
\gdef\SetFigFont#1#2#3#4#5{%
  \reset@font\fontsize{#1}{#2pt}%
  \fontfamily{#3}\fontseries{#4}\fontshape{#5}%
  \selectfont}%
\fi\endgroup%
\begin{picture}(984,1674)(2314,-1588)
\end{picture}%
} \Big) \nu , $$
where multiplication by $\nu$ is on the right strand.

But $Z_2 \Big ( \raisebox{-0.15 in}{\begin{picture}(0,0)%
\includegraphics{pic36.pstex}%
\end{picture}%
\setlength{\unitlength}{1105sp}%
\begingroup\makeatletter\ifx\SetFigFont\undefined%
\gdef\SetFigFont#1#2#3#4#5{%
  \reset@font\fontsize{#1}{#2pt}%
  \fontfamily{#3}\fontseries{#4}\fontshape{#5}%
  \selectfont}%
\fi\endgroup%
\begin{picture}(984,1674)(2314,-1588)
\end{picture}%
} \Big)=1$,
due to $\mu = 1$, as seen before. So we have
$$\raisebox{-0.4 in}{\input pic37.pstex_t} \thickspace =1.$$

Using that $\nu ^{-1}$ commutes with everything due to the 
Locality Lemma \ref{loc}, and multiplying by $b^{-1}$, we
get
$$\raisebox{-0.3 in}{\input pic38.pstex_t} \thickspace =
\thickspace \raisebox{-0.25 in}{\begin{picture}(0,0)%
\includegraphics{pic39.pstex}%
\end{picture}%
\setlength{\unitlength}{1421sp}%
\begingroup\makeatletter\ifx\SetFigFont\undefined%
\gdef\SetFigFont#1#2#3#4#5{%
  \reset@font\fontsize{#1}{#2pt}%
  \fontfamily{#3}\fontseries{#4}\fontshape{#5}%
  \selectfont}%
\fi\endgroup%
\begin{picture}(1592,2154)(949,-1888)
\put(1126,-736){\makebox(0,0)[lb]{\smash{{\SetFigFont{9}{10.8}{\rmdefault}{\mddefault}{\updefault}{\color[rgb]{0,0,0}$b^{-1}$}%
}}}}
\end{picture}%
} \medspace .$$

Since $a$ and $\nu ^{-1}$ commute, the lemma follows for all 
$x \in {\mathbb N}$ (we can ``pull each $b^{-1}$ through the 
vertex'' one by one, then group $a$'s and $\nu ^{-1}$'s together). 
Then by Taylor expansions, the statement follows for all $x$.

\end{proof}

\begin{proposition}
The corrections

$\lambda_{\downarrow \uparrow \uparrow} = a^{x-1}$

$u = \nu ^{-x}$

$Y_{\downarrow \downarrow \uparrow} = b^{-x}$

$n = \nu ^{x-1}$

make $Z_3$ a universal finite type invariant of knotted 
trivalent graphs that
commutes with the orientation change, edge delete and 
connected sum operation.
\end{proposition}

\begin{proof}
We need to check that the equations translated from moves 
$1$, $3$ and $4$ are satisfied to prove the invariance part:
The equation coming from move $1$ says $nu=\nu^{-1}$, 
which holds for the $n$ and $u$ above.

Moves $3$ and $4$ were translated in Proposition \ref{translation} to
$$\raisebox{-0.2 in}{\input pic20.pstex_t} \thickspace = \thickspace
\raisebox{-0.2 in}{\begin{picture}(0,0)%
\includegraphics{yvertex.pstex}%
\end{picture}%
\setlength{\unitlength}{1421sp}%
\begingroup\makeatletter\ifx\SetFigFont\undefined%
\gdef\SetFigFont#1#2#3#4#5{%
  \reset@font\fontsize{#1}{#2pt}%
  \fontfamily{#3}\fontseries{#4}\fontshape{#5}%
  \selectfont}%
\fi\endgroup%
\begin{picture}(1464,2289)(1159,-2473)
\put(1711,-1381){\makebox(0,0)[lb]{\smash{{\SetFigFont{12}{14.4}{\rmdefault}{\mddefault}{\updefault}{\color[rgb]{0,0,0}$Y$}%
}}}}
\end{picture}%
} \thickspace = \thickspace
\raisebox{-0.2 in}{\input pic21.pstex_t} \medspace .$$

This is satisfied since
$$\raisebox{-0.2 in}{\input pic20.pstex_t} \thickspace =
\thickspace \raisebox{-0.2 in}{\input pic40.pstex_t} \thickspace =
\thickspace \raisebox{-0.2 in}{\begin{picture}(0,0)%
\includegraphics{pic41.pstex}%
\end{picture}%
\setlength{\unitlength}{1145sp}%
\begingroup\makeatletter\ifx\SetFigFont\undefined%
\gdef\SetFigFont#1#2#3#4#5{%
  \reset@font\fontsize{#1}{#2pt}%
  \fontfamily{#3}\fontseries{#4}\fontshape{#5}%
  \selectfont}%
\fi\endgroup%
\begin{picture}(1990,2334)(874,-1888)
\put(1111,-691){\makebox(0,0)[lb]{\smash{{\SetFigFont{9}{10.8}{\rmdefault}{\mddefault}{\updefault}{\color[rgb]{0,0,0}$b^{-x}$}%
}}}}
\end{picture}%
} \thickspace =
\thickspace \raisebox{-0.2 in}{\begin{picture}(0,0)%
\includegraphics{yvertex.pstex}%
\end{picture}%
\setlength{\unitlength}{1421sp}%
\begingroup\makeatletter\ifx\SetFigFont\undefined%
\gdef\SetFigFont#1#2#3#4#5{%
  \reset@font\fontsize{#1}{#2pt}%
  \fontfamily{#3}\fontseries{#4}\fontshape{#5}%
  \selectfont}%
\fi\endgroup%
\begin{picture}(1464,2289)(1159,-2473)
\put(1711,-1381){\makebox(0,0)[lb]{\smash{{\SetFigFont{12}{14.4}{\rmdefault}{\mddefault}{\updefault}{\color[rgb]{0,0,0}$Y$}%
}}}}
\end{picture}%
} ,$$
and the mirror image done the same way.

The complete set of corrections follows from these
by inserting the appropriate orientation switches. This will
ensure that $Z_3$ commutes with the orientation switch operation.

The fact that $Z_3$ is a universal finite type invariant is true
by the exact same proof that applies to knots.

For edge delete and connected sum, the proof is the same as for
our first set of corrections.
\end{proof}

\begin{remark} \label{mo}
{\rm With the use of the symmetric corrections we get for
$x=1/2$, our construction coincides with that of \cite{murakami}.
(Except for the minor difference in the target space, as our skeleton 
graphs are not vertex-oriented. However, the definition of the framing at
vertices induces a vertex orientation, so the difference is insignificant.)}
\end{remark}

Throughout the above calculations and proofs, we had always assumed that
the re-normalization scale $\mu$ was chosen to be $1$. However,
when we change the scale from $\mu$ to $\mu'$, all that happens 
is that factors of
$exp \Big( \frac {ln(\mu'/\mu) t_{12}}{2\pi i} \Big)$ will be 
placed at $\lambda$-vertices and maxima, while the reciprocal 
is placed at $Y$-vertices and minima. Therefore, if the scale is
chosen differently, we can account for this by 
multiplying the correction terms with the inverses of these factors.

\subsection{Re-normalizing unzip}
There is one shortcoming of $Z_3$ still:
it fails to commute with the unzip operation. (As before, it is
enough to restrict our attention to vertical unzip and delete.)

Commuting with 
unzip would require that $Y \lambda =1$, in other words, after 
we unzip a vertical edge, we would want the top and the bottom 
corrections to cancel each other out. (As we have discussed before,
the chords ending on the unzipped edge come in sums of pairs, and 
$\lambda$ and $Y$ commute with these by the Locality Lemma \ref{loc},
so they could indeed cancel each other out.)

However, any set of corrections that makes $Z_3$ a knotted graph
invariant can either let it commute with the edge delete operation or 
with the unzip, but not both. This is easy to show:

As said above, for $Z_3$ to commute with unzip, we
need $$\raisebox{-0.2 in}{\input pic42.pstex_t} \medspace = 1.$$
Since $Y$ has an inverse in ${\mathcal A}_2$ (otherwise it couldn't 
be a correction), this inverse has to be $\lambda$. 

Therefore, we have
$$1= \medspace \raisebox{-0.3 in}{\input pic43.pstex_t} \medspace =
\medspace \raisebox{-0.25 in}{\input pic44.pstex_t} \medspace =
\medspace \raisebox{-0.25 in}{\begin{picture}(0,0)%
\includegraphics{pic45.pstex}%
\end{picture}%
\setlength{\unitlength}{1500sp}%
\begingroup\makeatletter\ifx\SetFigFont\undefined%
\gdef\SetFigFont#1#2#3#4#5{%
  \reset@font\fontsize{#1}{#2pt}%
  \fontfamily{#3}\fontseries{#4}\fontshape{#5}%
  \selectfont}%
\fi\endgroup%
\begin{picture}(1916,1824)(1068,-2173)
\put(1666,-841){\makebox(0,0)[lb]{\smash{{\SetFigFont{10}{12.0}{\rmdefault}{\mddefault}{\updefault}{\color[rgb]{0,0,0}$\nu^{-1}$}%
}}}}
\end{picture}%
} \medspace ,$$
a contradiction, where the second equality is due to the assumption that
$Z_3$ commutes with the edge delete operation: as discussed before,
this property is equivalent to
\begin{center}
$\raisebox{-0.15 in}{\begin{picture}(0,0)%
\includegraphics{pic22c.pstex}%
\end{picture}%
\setlength{\unitlength}{1539sp}%
\begingroup\makeatletter\ifx\SetFigFont\undefined%
\gdef\SetFigFont#1#2#3#4#5{%
  \reset@font\fontsize{#1}{#2pt}%
  \fontfamily{#3}\fontseries{#4}\fontshape{#5}%
  \selectfont}%
\fi\endgroup%
\begin{picture}(1284,1599)(1249,-1783)
\put(1696,-1366){\makebox(0,0)[lb]{\smash{{\SetFigFont{12}{14.4}{\rmdefault}{\mddefault}{\updefault}{\color[rgb]{0,0,0}$Y$}%
}}}}
\end{picture}%
} \thickspace
= \thickspace \raisebox{-0.1 in}{\begin{picture}(0,0)%
\includegraphics{pic23b.pstex}%
\end{picture}%
\setlength{\unitlength}{1579sp}%
\begingroup\makeatletter\ifx\SetFigFont\undefined%
\gdef\SetFigFont#1#2#3#4#5{%
  \reset@font\fontsize{#1}{#2pt}%
  \fontfamily{#3}\fontseries{#4}\fontshape{#5}%
  \selectfont}%
\fi\endgroup%
\begin{picture}(1554,1089)(1174,-1723)
\put(1786,-1576){\makebox(0,0)[lb]{\smash{{\SetFigFont{12}{14.4}{\rmdefault}{\mddefault}{\updefault}{\color[rgb]{0,0,0}$u$}%
}}}}
\end{picture}%
} \medspace , \thickspace \thickspace
\raisebox{-0.15 in}{\begin{picture}(0,0)%
\includegraphics{pic24c.pstex}%
\end{picture}%
\setlength{\unitlength}{1381sp}%
\begingroup\makeatletter\ifx\SetFigFont\undefined%
\gdef\SetFigFont#1#2#3#4#5{%
  \reset@font\fontsize{#1}{#2pt}%
  \fontfamily{#3}\fontseries{#4}\fontshape{#5}%
  \selectfont}%
\fi\endgroup%
\begin{picture}(1284,1599)(1249,-3358)
\put(1726,-2506){\makebox(0,0)[lb]{\smash{{\SetFigFont{11}{13.2}{\rmdefault}{\mddefault}{\updefault}{\color[rgb]{0,0,0}$\lambda$}%
}}}}
\end{picture}%
} \thickspace
= \thickspace \raisebox{-0.1 in}{\begin{picture}(0,0)%
\includegraphics{pic25b.pstex}%
\end{picture}%
\setlength{\unitlength}{1500sp}%
\begingroup\makeatletter\ifx\SetFigFont\undefined%
\gdef\SetFigFont#1#2#3#4#5{%
  \reset@font\fontsize{#1}{#2pt}%
  \fontfamily{#3}\fontseries{#4}\fontshape{#5}%
  \selectfont}%
\fi\endgroup%
\begin{picture}(1554,1089)(1174,-2788)
\put(1786,-2071){\makebox(0,0)[lb]{\smash{{\SetFigFont{12}{14.4}{\rmdefault}{\mddefault}{\updefault}{\color[rgb]{0,0,0}$n$}%
}}}}
\end{picture}%
} \thickspace ,$
\end{center}
which implies the second equality.

In the argument above, we're ignoring edge orientation issues for
simplicity.
The edge orientations compatible with unzip are not the same as the
ones compatible with edge deletion, so, strictly speaking, we're
not talking about the same $\lambda$ and $Y$. However, they only differ
by orientation switches, and orientation switches commute with taking 
inverses, so this doesn't interfere with the proof.

\smallskip
Now that we're convinced that this issue is unavoidable, we
fix it by re-normalizing the unzip operation on ${\mathcal A}$,
in other words we modify the algebra ${\mathcal A}$ slightly,
changing only the unzip operation.

The new unzip 
$\widetilde{u}: {\mathcal A}(\gamma) \to {\mathcal A}(u(\gamma))$
has to satisfy $\widetilde{u}(Z_3(\gamma))=Z_3(u(\gamma))$.
The obvious way to achieve this is to introduce a correction
that cancels out the error $Y \lambda$.

Take, for example, our first set of corrections. For (vertical) 
unzip to be defined, either all edges need to be oriented upwards,
or all downwards. In both cases $\lambda_{\downarrow \downarrow \downarrow}=
\lambda_{\uparrow \uparrow \uparrow}= a^{-1}$, and $Y=1$. 

Let $\widetilde{u}=i_a \circ u$, where by $i_a$ we mean ``inject a copy
of $a$ on the unzipped edge''. (This will also commute with the
pairs of chords ending on the edge, so it doesn't matter where 
we place it.)

\begin{remark}
{\rm In the case of Murakami and Ohtsuki's choice of corrections,
we have} $Y \lambda = b^{-1/2}a^{-1/2}$. {\rm The following fact is
proved by Le and Murakami in \cite{lemu}:}
$$\raisebox{-0.3 in}{\input pic46.pstex_t} \medspace  = 
\raisebox{-0.38 in}{\input pic47.pstex_t} \medspace . $$
\end{remark}
Let us sketch a simple proof using properties of $Z_2$ and $Z_3$:

By multiplying the tangles used to define $a$ and $b$, computing
$Z_2$ and unzipping a vertical edge with no chords, we get:
$$\raisebox{-0.3 in}{\input pic46.pstex_t} \thinspace =
Z_2 \Bigg( \raisebox{-0.3 in}{\begin{picture}(0,0)%
\includegraphics{pic51.pstex}%
\end{picture}%
\setlength{\unitlength}{2605sp}%
\begingroup\makeatletter\ifx\SetFigFont\undefined%
\gdef\SetFigFont#1#2#3#4#5{%
  \reset@font\fontsize{#1}{#2pt}%
  \fontfamily{#3}\fontseries{#4}\fontshape{#5}%
  \selectfont}%
\fi\endgroup%
\begin{picture}(882,1274)(3760,-1373)
\end{picture}%
 } \Bigg) .$$
We can produce this graph by an unzip: $\raisebox{-0.3 in}{\begin{picture}(0,0)%
\includegraphics{pic51.pstex}%
\end{picture}%
\setlength{\unitlength}{2605sp}%
\begingroup\makeatletter\ifx\SetFigFont\undefined%
\gdef\SetFigFont#1#2#3#4#5{%
  \reset@font\fontsize{#1}{#2pt}%
  \fontfamily{#3}\fontseries{#4}\fontshape{#5}%
  \selectfont}%
\fi\endgroup%
\begin{picture}(882,1274)(3760,-1373)
\end{picture}%
} \medspace=
u \Bigg( \raisebox{-0.3 in}{\begin{picture}(0,0)%
\includegraphics{pic52.pstex}%
\end{picture}%
\setlength{\unitlength}{2052sp}%
\begingroup\makeatletter\ifx\SetFigFont\undefined%
\gdef\SetFigFont#1#2#3#4#5{%
  \reset@font\fontsize{#1}{#2pt}%
  \fontfamily{#3}\fontseries{#4}\fontshape{#5}%
  \selectfont}%
\fi\endgroup%
\begin{picture}(854,1644)(2106,-1933)
\end{picture}%
} \Bigg),$ but
this not being a vertical unzip, it does not commute with $Z_2$,
so let us use $Z_3$ with our first set of corrections and the 
re-normalized version of unzip:
\begin{center}
$Z_3 \Bigg( \raisebox{-0.3 in}{\begin{picture}(0,0)%
\includegraphics{pic51.pstex}%
\end{picture}%
\setlength{\unitlength}{2605sp}%
\begingroup\makeatletter\ifx\SetFigFont\undefined%
\gdef\SetFigFont#1#2#3#4#5{%
  \reset@font\fontsize{#1}{#2pt}%
  \fontfamily{#3}\fontseries{#4}\fontshape{#5}%
  \selectfont}%
\fi\endgroup%
\begin{picture}(882,1274)(3760,-1373)
\end{picture}%
} \Bigg)=
\widetilde{u} Z_3 \Bigg( \raisebox{-0.3 in}{\begin{picture}(0,0)%
\includegraphics{pic52.pstex}%
\end{picture}%
\setlength{\unitlength}{2052sp}%
\begingroup\makeatletter\ifx\SetFigFont\undefined%
\gdef\SetFigFont#1#2#3#4#5{%
  \reset@font\fontsize{#1}{#2pt}%
  \fontfamily{#3}\fontseries{#4}\fontshape{#5}%
  \selectfont}%
\fi\endgroup%
\begin{picture}(854,1644)(2106,-1933)
\end{picture}%
} \Bigg)
=\widetilde{u} \medspace \raisebox{-0.4 in}{\input pic53.pstex_t}= 
\medspace \raisebox{-0.35 in}{\begin{picture}(0,0)%
\includegraphics{pic54.pstex}%
\end{picture}%
\setlength{\unitlength}{1855sp}%
\begingroup\makeatletter\ifx\SetFigFont\undefined%
\gdef\SetFigFont#1#2#3#4#5{%
  \reset@font\fontsize{#1}{#2pt}%
  \fontfamily{#3}\fontseries{#4}\fontshape{#5}%
  \selectfont}%
\fi\endgroup%
\begin{picture}(1731,2201)(2539,-2879)
\put(2618,-1627){\makebox(0,0)[lb]{\smash{{\SetFigFont{8}{9.6}{\rmdefault}{\mddefault}{\updefault}{\color[rgb]{0,0,0}$u \left( \nu^{-1} \right)$}%
}}}}
\end{picture}%
}.$
\end{center}
For the second equality above, we use the fact that $Z_2$ 
of this graph is trivial (everything cancels out), so in $Z_3$
all we have is the corrections.

Now we can get $Z_2$ back from $Z_3$ by undoing the corrections:
$$\raisebox{-0.3 in}{\input pic46.pstex_t} \medspace =
Z_2 \Bigg( \raisebox{-0.3 in}{\begin{picture}(0,0)%
\includegraphics{pic51.pstex}%
\end{picture}%
\setlength{\unitlength}{2605sp}%
\begingroup\makeatletter\ifx\SetFigFont\undefined%
\gdef\SetFigFont#1#2#3#4#5{%
  \reset@font\fontsize{#1}{#2pt}%
  \fontfamily{#3}\fontseries{#4}\fontshape{#5}%
  \selectfont}%
\fi\endgroup%
\begin{picture}(882,1274)(3760,-1373)
\end{picture}%
} \Bigg)=
\medspace \raisebox{-0.4 in}{\input pic55.pstex_t} \medspace,$$
which completes the proof.

Since the element on the right side of the equality is central
(by the Locality Lemma \ref{loc}), this implies that
$a$ and $b$ commute, and hence $a^x b^x =(ab)^x$.
So the re-normalization in this case (i.e.
injecting $a^{1/2}b^{1/2}$ on the unzipped edge) has an
additional, different description: we achieve the same by
first injecting $\nu^{-1/2}$ on the edge, then unzipping,
and then injecting $\nu^{1/2}$ on both new edges.

Taking any member of the one-parameter family of corrections,
we can re-normalize unzip by planting $\lambda^{-1}Y^{-1}$ on
the pair of edges resulting from the unzip. This way, in all cases, 
$Z_3$ becomes an invariant that commutes with all four operations
on knotted trivalent graphs and ${\mathcal A}$, with this 
re-normalized version of $u$. 

\subsection{Parenthesized tangles and Drinfeld's associators}
Let us end with a sketch of how this construction relates to 
parenthesized tangles (a.k.a. q-tangles, see for example 
\cite{dbn3}, \cite{lemu3}) and Drinfeld's 
associators (see \cite{dbn4}, \cite{dri1}, \cite{dri2}).
There is an easy map $\alpha$ from parenthesized tangles to KTGs. We
join ends by vertices according to the parenthetization, as 
shown:
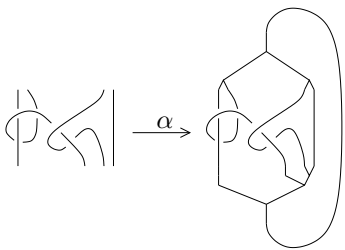
\begin{center}
\begin{picture}(0,0)%
\includegraphics{pic48.pstex}%
\end{picture}%
\setlength{\unitlength}{1579sp}%
\begingroup\makeatletter\ifx\SetFigFont\undefined%
\gdef\SetFigFont#1#2#3#4#5{%
  \reset@font\fontsize{#1}{#2pt}%
  \fontfamily{#3}\fontseries{#4}\fontshape{#5}%
  \selectfont}%
\fi\endgroup%
\begin{picture}(5299,3786)(850,-3304)
\put(3226,-1411){\makebox(0,0)[lb]{\smash{{\SetFigFont{10}{12.0}{\rmdefault}{\mddefault}{\updefault}{\color[rgb]{0,0,0}$\alpha$}%
}}}}
\end{picture}%

\end{center}
The parenthetization of the bottom of the tangle shown is
$(*(*(**)))$, while on the top it is $((**)(**))$. The ends
are joined by vertices accordingly.

Composition of parenthesized tangles translates to taking the
connected sum of KTGs, and then unzipping the middle edges until no more
unzips are possible. As an example, let's multiply the above tangle 
by its mirror image:
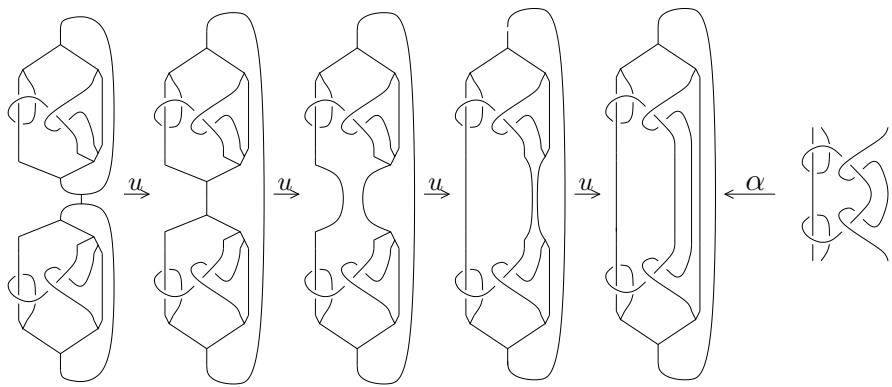
\begin{center}
\input pic49.pstex_t
\end{center}
Looking at the picture, we can convince ourselves that we get the
same result by first applying $\alpha$ (as seen on the previous figure),
taking the connected sum of the resulting KTG's, and unzipping the middle,
or by first multiplying the parenthesized tangles, and then applying
$\alpha$. 

The inverse of $\alpha$ amounts to unzipping all edges starting from the 
top to bottom one, until no trivalent vertices are left, and then cutting
the unzipped edges.

Since $Z_3$ commutes with the (re-normalized) unzip operation,
we see that in particular, 
$Z_3 \Big( \raisebox{-0.1 in}{\begin{picture}(0,0)%
\includegraphics{pic50.pstex}%
\end{picture}%
\setlength{\unitlength}{1381sp}%
\begingroup\makeatletter\ifx\SetFigFont\undefined%
\gdef\SetFigFont#1#2#3#4#5{%
  \reset@font\fontsize{#1}{#2pt}%
  \fontfamily{#3}\fontseries{#4}\fontshape{#5}%
  \selectfont}%
\fi\endgroup%
\begin{picture}(1224,924)(5089,-1423)
\end{picture}%
} \Big)$ will satisfy the
pentagon and hexagon equations (\cite{dbn4}, \cite{dri1}, \cite{dri2}), hence 
it is a construction of an associator.

\end{document}